\documentclass[12pt]{amsart}

\usepackage{amsmath,amssymb,amsfonts,mathtools}
\usepackage{xcolor}
\usepackage[colorlinks=true,citecolor=red,linkcolor=blue,urlcolor=blue,
            hypertexnames=false]{hyperref}

\allowdisplaybreaks
\numberwithin{equation}{section}

\newcommand{\R}{\mathbb{R}}

\newcommand{\BMO}{\operatorname{BMO}}
\newcommand{\BMOG}{\operatorname{BMO}_G}
\newcommand{\dist}{\operatorname{dist}}
\newcommand{\supp}{\operatorname{supp}}
\newcommand{\conv}{\operatorname{conv}}

\theoremstyle{plain}
\newtheorem{theorem}{Theorem}[section]
\newtheorem{proposition}[theorem]{Proposition}
\newtheorem{lemma}[theorem]{Lemma}
\newtheorem{corollary}[theorem]{Corollary}

\theoremstyle{definition}
\newtheorem{definition}[theorem]{Definition}

\theoremstyle{remark}
\newtheorem{remark}[theorem]{Remark}

\title[BMO classification for two-dimensional Dunkl kernels]
{BMO Classification for Two-Dimensional Dunkl Newton
and Green Kernels
}
\author{Ji Li}
\address{Ji Li, School of Mathematical and Physical Sciences,
	Macquarie University, NSW 2109, Australia}
\email{ji.li@mq.edu.au}

 \author{Lixin Yan}
\address{Lixin Yan, School of Mathematics,
	Sun Yat-sen University,
	Guangzhou, 510275,
	P.R.~China}
\curraddr{}
\email{mcsylx@mail.sysu.edu.cn}

\author{Huohao Zhang}
\address{Huohao Zhang, School of Mathematics,
	Sun Yat-sen University,
	Guangzhou, 510275,
	P.R.~China}
\email{zhanghh58@mail2.sysu.edu.cn}


\begin{document}
\begin{abstract}
Using Gaussian heat-kernel estimates, we classify planar Dunkl Newton kernels
in weighted Euclidean BMO and obtain the corresponding local classification for
unit-ball Green kernels.  For atomic Newton potentials supported on a regular
reflection orbit, orbit BMO detects exactly whether the coefficients are constant
along the orbit.  We also prove an intrinsic Dunkl--CLMS theorem in the
heat-semigroup Hardy space and derive intrinsic and Euclidean-source
Newton--Wente estimates, extending the BMO--Hardy-space method of Chanillo and
Li to the Dunkl setting.  Further consequences include sharp local atomic
estimates, same-domain criteria, a localized $L^1$-to-BMO bound, and a
Brezis--Merle-type estimate.
\end{abstract}
\maketitle

\section{Introduction and statement of the main results}
\label{sec:introduction}

\subsection{The classical model}

Consider a second-order divergence-form operator in the plane,
\[
 Lu=-\sum_{i,j=1}^2D_i\bigl(a_{ij}D_ju\bigr),
\]
where $A=(a_{ij}(x))_{1\leq i,j\leq2}$ is real, symmetric, and
measurable.  Assume that there is a constant $\lambda>0$ such that
\begin{equation}\label{e1.1}
 A(x)\xi\cdot\xi\geq\lambda|\xi|^2
 \qquad\text{and}\qquad
 |A(x)\xi\cdot\zeta|\leq\lambda^{-1}|\xi||\zeta|
\end{equation}
for all $\xi,\zeta\in\mathbb R^2$ and almost every $x\in\mathbb R^2$.

Chanillo and Li \cite{ChanilloLi1992} organized their study of this
operator around one kernel result and two applications.  If $G_x$ is the
global Green function of $L$ with pole at $x$, defined up to an additive
constant, their first result states that, for every ball
$B=B_R(x_0)\subset\mathbb R^2$ and every $1<p<2$,
\begin{equation}\label{eq:CL-gradient-estimate}
 R\left(\frac1{|B|}\int_B|\nabla G_x(y)|^p\,dy\right)^{1/p}
 \leq C(\lambda,p).
\end{equation}
The Poincar\'e--Sobolev inequality then gives
\begin{equation}\label{eq:CL-BMO-conclusion}
 \|G_x\|_{\BMO(\mathbb R^2)}\leq C(\lambda),
\end{equation}
uniformly in the pole $x$.

Their second result is a Wente-type application of this BMO estimate.  Let
$\Omega\subset\mathbb R^2$ be a bounded domain with $C^1$ boundary.  If
$u,v\in H^1(\Omega)$ and $\varphi\in W_0^{1,1}(\Omega)$ solves
\[
 \begin{cases}
 L\varphi=u_{x_1}v_{x_2}-u_{x_2}v_{x_1}&\text{in }\Omega,\\
 \varphi=0&\text{on }\partial\Omega,
 \end{cases}
\]
then
\[
 \|\varphi\|_{L^\infty(\Omega)}
 +\|\nabla\varphi\|_{L^2(\Omega)}
 \leq C(\lambda)
 \|\nabla u\|_{L^2(\Omega)}
 \|\nabla v\|_{L^2(\Omega)}.
\]
The Hardy-space estimate for the Jacobian and $H^1$--BMO duality are the
main link between the Green-kernel estimate and this conclusion.  Their
third result is a Brezis--Merle-type exponential estimate for the solution
of $L\varphi=f$ with $f\in L^1(\Omega)$.  Thus the main line of their paper
is
\[
 \begin{gathered}
 \text{kernel bounds}\quad\Longrightarrow\quad\text{BMO},\\
 \text{BMO}\quad\Longrightarrow\quad
 \text{Wente and exponential estimates}.
 \end{gathered}
\]
The whole-plane Laplacian estimate is due to Wente \cite{Wente}.
For later extensions to complex coefficients and to elliptic systems and
mixed boundary problems, see
\cite{AuscherMcIntoshTchamitchian1998,TaylorKimBrown2013}.

\subsection{The Dunkl setting and main results}

Dunkl operators are differential--reflection versions of directional
derivatives associated with a finite root system.  Their harmonic analysis
has many Euclidean features, but reflection orbits produce new local
geometry and new singularities.  We refer to \cite{Dunkl1989,ADH} for the
basic setting.  The Dunkl Newton kernel and the Green kernel of the unit
ball were constructed and studied by Graczyk, Luks, and R\"osler
\cite{GLR}.

We remove roots with zero multiplicity and denote the remaining root
system again by $R$.  Thus $k(\alpha)>0$ for every $\alpha\in R$.  Let $G$
be the corresponding reflection group, and put
\[
 w_k(x)=\prod_{\alpha\in R_+}|\langle\alpha,x\rangle|^{2k(\alpha)},
 \qquad d\mu_k(x)=w_k(x)\,dx,
 \qquad \gamma=\sum_{\alpha\in R_+}k(\alpha).
\]
The reflecting lines and the regular set are
\[
 \mathcal H=\bigcup_{\alpha\in R}\alpha^\perp,
 \qquad
 \mathbb R^2_{\mathrm{reg}}=\mathbb R^2\setminus\mathcal H.
\]
The connected components of $\mathbb R^2_{\mathrm{reg}}$ are the open
Weyl chambers.  If $\Gamma_k(t,x,y)$ is the Dunkl heat kernel and
$Q=2+2\gamma>2$, the global Newton kernel is
\[
 N_k(x,y)=\int_0^\infty\Gamma_k(t,x,y)\,dt.
\]
We write $G_k$ for the Dirichlet Green kernel of the unit ball $B$.
Unless stated otherwise, $\BMO(\mu_k)$ is defined using Euclidean balls.

The first question is whether the uniform BMO conclusion
\eqref{eq:CL-BMO-conclusion} survives in this setting.  It does not.  The
answer depends on the position of the pole.  A regular pole has a
logarithmic singularity, while a pole on a reflecting line produces power
growth of small-ball averages.  Our first main result gives the exact
classification.

\begin{theorem}\label{thm:main-classification}
Assume $\gamma>0$.  Let $N_k$ be the global Dunkl Newton kernel and $G_k$
the Dirichlet Green kernel of the unit ball $B$.  Then:
\begin{enumerate}
\item[(a)] $N_k(\cdot,0)\notin\BMO(\mu_k)$, and
$G_k(\cdot,0)\notin\BMO_{\mathrm{loc}}(B,\mu_k)$.
\item[(b)] If $y\in\mathcal H\setminus\{0\}$, then
$N_k(\cdot,y)\notin\BMO(\mu_k)$; if also $y\in B$, then
$G_k(\cdot,y)\notin\BMO_{\mathrm{loc}}(B,\mu_k)$.
\item[(c)] If $y\in\mathbb R^2_{\mathrm{reg}}$, then
$N_k(\cdot,y)\in\BMO(\mu_k)$; if also $y\in B$, then
$G_k(\cdot,y)\in\BMO_{\mathrm{loc}}(B,\mu_k)$.
\end{enumerate}
\end{theorem}

For a regular pole, the proof uses the Gaussian size and spatial regularity
estimates for the Dunkl heat kernel.  Their time integrals give a
logarithmic majorant and a direct BMO bound.  The bound is uniform when the
pole ranges over a compact subset of $\mathbb R^2_{\mathrm{reg}}$.  At the
origin and on a reflecting line, a lower heat-kernel estimate gives power
growth of small-ball averages, which is not compatible with BMO.  The
Green-kernel statements follow from the exact Newton--Green decomposition
in \cite[Theorem~3.1]{GLR}.

Section~\ref{sec:applications} gives two first applications.  For a regular
point $a$ and coefficients $(c_b)_{b\in\mathcal O(a)}$, the atomic Newton
potential
\[
 U_c(x)=\sum_{b\in\mathcal O(a)}c_bN_k(x,b)
\]
always belongs to Euclidean BMO, while it belongs to orbit BMO if and only
if the coefficients are constant on the orbit.  The same section also
shows that, for a bounded open set $\Omega$, the local atomic Newton map
\[
 N_\Omega:H^1_{\mathrm{at}}(\Omega,\mu_k)
 \longrightarrow L^\infty(\Omega)
\]
is bounded if and only if
\[
 \dist(\overline\Omega,\mathcal H)>0.
\]
This geometric condition also governs the Newton--Wente estimates below.

We next turn to the analogue of the second step in Chanillo and Li's
argument.  Write
\[
 \nabla_k u=(T_{e_1}u,T_{e_2}u),
 \qquad
 \mathcal J_k(u,v)
 =T_{e_1}u\,T_{e_2}v-T_{e_2}u\,T_{e_1}v.
\]
Let $H_k^1$ be the Hardy space associated with the Dunkl heat semigroup.
We prove the Dunkl--CLMS estimate
\begin{equation}\label{eq:intro-dunkl-CLMS}
 \|\mathcal J_k(u,v)\|_{H_k^1}
 \leq C_k
 \|\nabla_k u\|_{L^2(d\mu_k)}
 \|\nabla_k v\|_{L^2(d\mu_k)},
\end{equation}
for $u,v\in C_c^\infty(\mathbb R^2)$.
At regular evaluation points, Theorem~\ref{thm:main-classification} has a
compact-uniform BMO form.  Combining this estimate with
\eqref{eq:intro-dunkl-CLMS} and $H_k^1$--BMO duality gives the positive
Newton--Wente bound.  A concentrating construction near the reflecting
lines gives the converse.

For $u\in C_c^\infty(\Omega\setminus\mathcal H)$, let $\widetilde u$ be
its zero extension to $\mathbb R^2$.  Let
$W^{1,2}_{k,0}(\Omega)$ be the completion of this test class in the norm
\[
 \left(
 \|\widetilde u\|_{L^2(d\mu_k)}^2
 +\|\nabla_k\widetilde u\|_{L^2(d\mu_k)}^2
 \right)^{1/2}.
\]
For test functions $u,v,F,G\in C_c^\infty(\Omega\setminus\mathcal H)$,
define
\[
 \mathcal W_{\Omega,k}(u,v)(x)
 =\int_\Omega N_k(x,y)
 \mathcal J_k(\widetilde u,\widetilde v)(y)\,d\mu_k(y)
\]
and
\[
 W_\Omega(F,G)(x)
 =\int_\Omega N_k(x,y)
 \bigl(F_{y_1}G_{y_2}-F_{y_2}G_{y_1}\bigr)(y)\,dy.
\]
The first potential has an intrinsic Dunkl source.  The second has a
Euclidean Jacobian source and is the closer analogue of the classical
Wente potential.
In statement (iii) below, a bounded bilinear extension means a bounded
bilinear map on $W_0^{1,2}(\Omega)\times W_0^{1,2}(\Omega)$ that agrees with
$W_\Omega$ on the displayed test class.  No density assertion is intended
when $\Omega$ meets $\mathcal H$.

\begin{theorem}[Sharp Newton--Wente criterion]
\label{thm:intro-wente}
Assume $\gamma>0$, and let $\Omega\subset\mathbb R^2$ be a bounded domain.
The following statements are equivalent:
\begin{enumerate}
\item[(i)] $\dist(\overline\Omega,\mathcal H)>0$;
\item[(ii)] $\mathcal W_{\Omega,k}$ admits a bounded bilinear extension
\[
 W^{1,2}_{k,0}(\Omega)\times W^{1,2}_{k,0}(\Omega)
 \longrightarrow L^\infty(\Omega);
\]
\item[(iii)] $W_\Omega$ admits a bounded bilinear extension
\[
 W_0^{1,2}(\Omega)\times W_0^{1,2}(\Omega)
 \longrightarrow L^\infty(\Omega).
\]
\end{enumerate}
When these conditions hold,
\[
 \|\mathcal W_{\Omega,k}(u,v)\|_{L^\infty(\Omega)}
 \leq C_{\Omega,k}
 \|\nabla_k\widetilde u\|_{L^2(d\mu_k)}
 \|\nabla_k\widetilde v\|_{L^2(d\mu_k)}
\]
and
\[
 \|W_\Omega(F,G)\|_{L^\infty(\Omega)}
 \leq C_{\Omega,k}
 \|\nabla F\|_{L^2(\Omega)}
 \|\nabla G\|_{L^2(\Omega)}.
\]
\end{theorem}

\begin{remark}
The corresponding uniform
estimates on all of $\mathbb R^2$ fail when $\gamma>0$, as follows from the
concentrating construction near the reflecting lines used in the proof of
Proposition~\ref{prop:intrinsic-Wente-obstruction}. This motivates restricting attention to bounded domains separated from the reflecting lines. Indeed, if $\Omega$ is bounded and $\dist(\overline\Omega,\mathcal H)>0$, the relevant poles remain in a compact subset of $\mathbb R^2_{\mathrm{reg}}$, where the Newton-kernel BMO bounds are uniform.
\end{remark}

Theorem~\ref{thm:intro-wente} concerns the global Newton kernel restricted
to $\Omega$, rather than the Dirichlet Green kernel.  For smooth data, the
propositions used in its proof also identify the corresponding
distributional equations.  The proof also gives the global
intrinsic estimate on every compact set
$K\Subset\mathbb R^2_{\mathrm{reg}}$.

The last main result is a separate application of the uniform regular-pole
estimates.  If
$\Omega\Subset\mathbb R^2_{\mathrm{reg}}$ and
$f\in L^1(\Omega,d\mu_k)$, put
\[
 N_\Omega f(x)=\int_\Omega N_k(x,y)f(y)\,d\mu_k(y),
 \qquad x\in\mathbb R^2.
\]
The uniform BMO estimate gives an $L^1$-to-BMO bound, while the stronger
pointwise logarithmic estimate gives normalized exponential
integrability.

\begin{theorem}[Localized BMO and Brezis--Merle-type estimate]
\label{thm:intro-brezis-merle}
Let $\Omega\Subset\mathbb R^2_{\mathrm{reg}}$.  There are constants
$c_{\Omega,k},C_{\Omega,k}>0$ such that
\[
 \|N_\Omega f\|_{\BMO(\mu_k)}
 \leq C_{\Omega,k}\|f\|_{L^1(\Omega,d\mu_k)}
\]
for every $f\in L^1(\Omega,d\mu_k)$, and, for every nonzero such $f$,
\[
 \int_\Omega
 \exp\left(
 c_{\Omega,k}
 \frac{|N_\Omega f(x)|}
 {\|f\|_{L^1(\Omega,d\mu_k)}}
 \right)d\mu_k(x)
 \leq C_{\Omega,k}.
\]
Consequently, for every $f\in L^1(\Omega,d\mu_k)$ and every $\beta>0$,
\[
 \exp\bigl(\beta|N_\Omega f|\bigr)
 \in L^1(\Omega,d\mu_k).
\]
\end{theorem}

The proof of the exponential estimate uses the uniform logarithmic
majorant, positivity of the Newton kernel, and Jensen's inequality.  We do
not claim a sharp exponential constant.

The rest of the paper is organized as follows.
Section~\ref{sec:dunkl-bmo} recalls the Dunkl setting, proves
Theorem~\ref{thm:main-classification}, and gives the refined estimate at
reflected orbit points.  Section~\ref{sec:applications} proves the
orbit-balance theorem, off-diagonal regularity, and the local atomic
results.  Section~\ref{sec:wente} proves the Dunkl--CLMS theorem and the
intrinsic and Euclidean-source Newton--Wente estimates, including
Theorem~\ref{thm:intro-wente}.  Section~\ref{sec:localized-newton} proves
the localized $L^1$-to-BMO bound and
Theorem~\ref{thm:intro-brezis-merle}.

\section{Dunkl setting and BMO classification}
\label{sec:dunkl-bmo}

\subsection{Basic Dunkl notation in the plane}

Let $R\subset\R^2\setminus\{0\}$ be a reduced finite root system,
let $R_+$ be a positive subsystem, and let $G$ be the associated reflection group. For $\alpha\in R$,
reflection across $H_\alpha=\alpha^\perp$ is given by
\[
 \sigma_\alpha x=x-2\frac{\langle x,\alpha\rangle}{|\alpha|^2}\alpha.
\]
For $E\subset\R^2$, write $\mathcal O(E)=\bigcup_{\sigma\in G}\sigma(E)$; in particular,
$\mathcal O(x)=\{\sigma(x):\sigma\in G\}$.

A multiplicity is a $G$-invariant map $k:R\to(0,\infty)$. As explained in the introduction, this
entails no loss of generality after the roots of zero multiplicity have been discarded. Put
\[
 \gamma=\sum_{\alpha\in R_+}k(\alpha),
 \qquad Q=2+2\gamma,
\]
and define
\[
 \mathcal H=\bigcup_{\alpha\in R}H_\alpha,
 \qquad \R^2_{\mathrm{reg}}=\R^2\setminus\mathcal H.
\]
We call each connected component of $\R^2_{\mathrm{reg}}$ an open Weyl
chamber, and we call its closure a closed Weyl chamber.
The Dunkl weight and measure are
\[
 w_k(x)=\prod_{\alpha\in R_+}|\langle\alpha,x\rangle|^{2k(\alpha)},
 \qquad d\mu_k(x)=w_k(x)\,dx.
\]
All implicit constants may depend on the fixed root system $R$ and multiplicity $k$.
Several sources used below state their results for normalized root systems,
with $|\alpha|^2=2$.  This entails no restriction here.  Set
\[
 \begin{aligned}
  \widetilde R
  &=\left\{\frac{\sqrt2\,\alpha}{|\alpha|}:\alpha\in R\right\},\\
  \widetilde k\left(\frac{\sqrt2\,\alpha}{|\alpha|}\right)
  &=k(\alpha),\\
  a_{R,k}
  &=\prod_{\alpha\in R_+}
    \left(\frac{\sqrt2}{|\alpha|}\right)^{2k(\alpha)}.
 \end{aligned}
\]
Then $\widetilde R$ is normalized and has the same reflection group as $R$,
and $d\mu_{\widetilde k}=a_{R,k}\,d\mu_k$.
The standard volume estimate \cite[(2.3)]{DHatomic} is
\begin{equation}\label{eq:volume-estimate}
 \mu_k(B(x,r))\simeq r^2\prod_{\alpha\in R_+}
 \bigl(|\langle\alpha,x\rangle|+r\bigr)^{2k(\alpha)}.
\end{equation}
In particular, $\mu_k$ is doubling,
\[
 w_k(rx)=r^{2\gamma}w_k(x),
 \qquad \mu_k(B(0,r))=r^Q\mu_k(B(0,1)),
\]
and $\mu_k(B(x,r))\gtrsim r^Q$ for every $x$ and $r>0$.

If $y\neq0$ lies on the reflecting line $H$ and $\kappa_H=k(\alpha)>0$ for $H=H_\alpha$, then locally
\begin{equation}\label{eq:local-weight-line}
 w_k(y+u)=c_y(u)|u_\perp|^{2\kappa_H},
\end{equation}
where $c_y$ is positive and smooth. If $y\in\R^2_{\mathrm{reg}}$, the weight is bounded above and below by positive
constants near $y$.

Unless explicitly stated otherwise, $\BMO(\mu_k)$ below is the Euclidean-ball space
\[
 \|f\|_{\BMO(\mu_k)}=
 \sup_B\frac1{\mu_k(B)}\int_B|f-f_B|\,d\mu_k,
 \qquad
 f_B=\frac1{\mu_k(B)}\int_Bf\,d\mu_k.
\]
The displayed quantity is a seminorm; whenever a BMO space is used as a
normed space, functions differing by a constant are identified.
For an open set $U\subset\R^2$, we write
$f\in\BMO_{\mathrm{loc}}(U,\mu_k)$ if
$f\in L^1_{\mathrm{loc}}(U,d\mu_k)$ and the supremum in the preceding
display, restricted to Euclidean balls $B\subset B_0$, is finite for every
Euclidean ball $B_0\Subset U$.
This weighted Euclidean BMO space is used throughout. The orbit-metric space $\BMOG$ is
introduced in Section~\ref{subsec:orbit-balance}.

\subsection{Dunkl operators and the Dunkl Laplacian}

The Dunkl operators were introduced
in \cite{Dunkl1989}. For $\xi\in\R^2$, they are given by
\[
 T_\xi f(x)=\partial_\xi f(x)+
 \sum_{\alpha\in R_+}k(\alpha)\langle\alpha,\xi\rangle
 \frac{f(x)-f(\sigma_\alpha x)}{\langle\alpha,x\rangle}.
\]
With respect to an orthonormal basis $(e_1,e_2)$, the Dunkl Laplacian is
\[
 \Delta_k=T_{e_1}^2+T_{e_2}^2.
\]
For sufficiently smooth functions one has the expanded form
\[
 \Delta_k f(x)=\Delta f(x)+\sum_{\alpha\in R_+}k(\alpha)
 \left(
 \frac{2\langle\nabla f(x),\alpha\rangle}{\langle\alpha,x\rangle}
 -|\alpha|^2\frac{f(x)-f(\sigma_\alpha x)}{\langle\alpha,x\rangle^2}
 \right).
\]
The operator is symmetric with respect to $d\mu_k$ on natural test classes.
Since the defining quotient in $T_\xi$ is unchanged when a root is multiplied
by a positive scalar,
\[
 T_\xi^{\widetilde R,\widetilde k}=T_\xi^{R,k},
 \qquad
 \Delta_{\widetilde k}=\Delta_k.
\]

The intertwining operator $V_k$ is characterized on polynomials by
\[
 T_\xi V_k=V_k\partial_\xi,
 \qquad V_k1=1.
\]
Following \cite{GLR}, set $C(y)=\conv\mathcal O(y)$. R\"osler's positivity theorem \cite{RoslerPositivity} gives a probability
measure $\mu_y^k$, supported in $C(y)$, such that
\[
 V_kf(y)=\int_{C(y)}f(z)\,d\mu_y^k(z).
\]
The fine structure of $\mu_y^k$ is one of the main difficulties in obtaining sharp Green estimates for
general root systems.

\subsection{Dunkl heat kernel in the plane}

We use the notation $\Gamma_k(t,x,y)$ of \cite{GLR}; the same
kernel is denoted by $h_t(x,y)$ in \cite{ADH,DHheat}. For fixed $x\in\R^2$, the Dunkl kernel $y\mapsto E(x,y)$ is the
unique solution of
\[
\begin{cases}
	T_\xi f=\langle\xi,x\rangle f,& \forall\xi\in\R^2,\\
	f(0)=1.
\end{cases}
\]
The heat semigroup has the kernel representation \cite[Section~4]{RoslerHeat}
\[
H_tf(x)=e^{t\Delta_k}f(x)=\int_{\R^2}\Gamma_k(t,x,y)f(y)\,d\mu_k(y),
\]
where
\[
\Gamma_k(t,x,y)=c_k^{-1}(2t)^{-1-\gamma}
e^{-(|x|^2+|y|^2)/(4t)}
E\!\left(\frac{x}{\sqrt{2t}},\frac{y}{\sqrt{2t}}\right).
\]
Here $c_k=\int_{\R^2}e^{-|x|^2/2}\,d\mu_k(x)$.  Since
$c_{\widetilde k}=a_{R,k}c_k$, the heat kernel relative to
$d\mu_{\widetilde k}$ satisfies
\[
 \Gamma_k(t,x,y)=a_{R,k}\Gamma_{\widetilde k}(t,x,y).
\]
Consequently, the heat-kernel estimates in \cite{ADH,DHheat} transfer to the
present convention.  The same holds for the Hardy-space identification in
\cite{DHatomic}, the Riesz-commutator bounds in \cite{DHcommutators}, and the
fundamental-solution result in \cite{GallardoRejeb}.  The heat semigroup and
Riesz transforms are unchanged, while
\[
 \|f\|_{L^p(d\mu_{\widetilde k})}
 =a_{R,k}^{1/p}\|f\|_{L^p(d\mu_k)},
 \qquad
 \|b\|_{\BMO(\mu_{\widetilde k})}
 =\|b\|_{\BMO(\mu_k)}
\]
for $1\leq p<\infty$.  The nontangential maximal functions also coincide, so
the corresponding Hardy $H^1$ norms differ by the factor $a_{R,k}$.
Set
\[
d(x,y)=\min_{\sigma\in G}|x-\sigma(y)|.
\]
We use only the following consequences of the basic estimates in
\cite[Theorems~4.1 and~4.4]{ADH}.

\begin{theorem}\label{thm:basic-heat-kernel}
	There are $c,C>0$ such that, for $x,y\in\R^2$ and $t>0$,
	\begin{equation}\label{eq:basic-heat-size}
		\frac{c}{\mu_k(B(x,\sqrt t))}e^{-C|x-y|^2/t}
		\leq \Gamma_k(t,x,y)
		\leq \frac{C}{\mu_k(B(x,\sqrt t))}e^{-cd(x,y)^2/t}.
	\end{equation}
	Moreover, if $|y-y'|<\sqrt t/2$, then
	\begin{equation}\label{eq:basic-heat-regularity}
		|\Gamma_k(t,x,y)-\Gamma_k(t,x,y')|
		\leq C\frac{|y-y'|}{\sqrt t}\frac1{\mu_k(B(x,\sqrt t))}
		e^{-cd(x,y)^2/t}.
	\end{equation}
	The same estimate holds in the first spatial variable by symmetry. Here $G$ and $d$ are the
	reflection group and orbit distance fixed above.
\end{theorem}

\subsection{Newton and Green kernels}

From now on we assume $Q=2+2\gamma>2$. The global
Newton kernel is
\[
 N_k(x,y)=\int_0^\infty\Gamma_k(t,x,y)\,dt.
\]
The symmetry and $G$-covariance of the heat kernel pass to the Newton kernel:
\[
 N_k(x,y)=N_k(y,x),
 \qquad N_k(\sigma x,\sigma y)=N_k(x,y),
 \qquad \sigma\in G.
\]
In dimension two, the representation of Graczyk--Luks--R\"osler \cite{GLR} is
\[
 N_k(x,y)=C_k\int_{C(y)}
 \frac{d\mu_y^k(z)}{\bigl(|x|^2+|y|^2-2\langle x,z\rangle\bigr)^\gamma}.
\]
Here
\[
 C_k=\frac1{2\gamma d_k},
 \qquad d_k=\int_{S^1}w_k(\theta)\,d\sigma(\theta).
\]
Since $Q>2$, the defining heat integral converges at infinity, so $N_k$ is a genuine zero-
resolvent kernel.  Integrating the heat-kernel scaling above gives
$N_k=a_{R,k}N_{\widetilde k}$.  Hence
\cite[Theorem~6.1]{GallardoRejeb}, together with the symmetry of $N_k$, gives
the distributional fundamental-solution identity in the present normalization:
\begin{equation}\label{eq:fundamental-solution}
 \int_{\R^2}N_k(x,y)(-\Delta_k\varphi)(x)\,d\mu_k(x)=\varphi(y),
 \qquad \varphi\in C_c^\infty(\R^2).
\end{equation}
Thus no large-time subtraction or normalization is involved.

If $y=0$, then $C(y)=\{0\}$ and $\mu_0^k=\delta_0$. Therefore
\begin{equation}\label{eq:newton-origin}
 N_k(x,0)=C_k|x|^{-2\gamma}.
\end{equation}

\subsubsection{The unit-ball Green kernel}

Let $B=\{x\in\R^2:|x|<1\}$. Graczyk--Luks--R\"osler \cite[Theorem~3.1]{GLR} prove the exact identity
\begin{equation}\label{eq:newton-green}
 G_k(x,y)=N_k(x,y)-K_k[N_k(\cdot,y)](x),
\end{equation}
where, in dimension two,
\begin{equation}\label{eq:green-correction}
 K_k[N_k(\cdot,y)](x)=C_k\int_{C(y)}
 \frac{d\mu_y^k(z)}{\bigl(1+|x|^2|y|^2-2\langle x,z\rangle\bigr)^\gamma}.
\end{equation}
For fixed $y\in B$, the correction in \eqref{eq:green-correction} is bounded and $\Delta_k$-harmonic in $B$. Consequently $G_k(\cdot,y)$
and $N_k(\cdot,y)$ have the same local BMO classification at the pole. This exact decomposition is
the only Green-kernel input needed below. At the origin it gives
\[
 G_k(x,0)=C_k\bigl(|x|^{-2\gamma}-1\bigr),
 \qquad x\in B\setminus\{0\}.
\]

\subsection{Regular poles: a basic heat-kernel proof}

\begin{proposition}\label{prop:regular-poles-bmo}
If $a\in\R^2_{\mathrm{reg}}$, then $N_k(\cdot,a)\in\BMO(\mu_k)$. Moreover, for every compact set
$K\Subset\R^2_{\mathrm{reg}}$,
\begin{equation}\label{eq:uniform-regular-poles-bmo}
 \sup_{a\in K}\|N_k(\cdot,a)\|_{\BMO(\mu_k)}<\infty.
\end{equation}
\end{proposition}

\begin{proof}
Since $a\in\R^2_{\mathrm{reg}}$, choose $\rho_a>0$ so that the balls $B(b,4\rho_a)$, $b\in\mathcal O(a)$, are pairwise disjoint
and stay a positive distance from $\mathcal H$. On their union, $w_k$ is bounded above and below by
positive constants.

By symmetry, the upper bound in \eqref{eq:basic-heat-size} and the volume estimate \eqref{eq:volume-estimate} give
\begin{equation}\label{eq:newton-log-majorant}
 0\leq N_k(y,a)\leq C_a\left(1+\log^+\frac{\rho_a}{d(y,a)}\right),
 \qquad d(y,a)=\min_{\sigma\in G}|y-\sigma(a)|.
\end{equation}
Indeed, the integral over $0<t<\rho_a^2$ is bounded by
\[
 C_a\int_0^{\rho_a^2}t^{-1}e^{-cd(y,a)^2/t}\,dt,
\]
while the integral over $\rho_a^2\leq t<\infty$ is finite because
$\mu_k(B(a,\sqrt t))\gtrsim t^{Q/2}$ and $Q>2$.

Let $B=B(z,r)$. First suppose $r<\rho_a/16$. If $B$ does not meet
$\bigcup_{b\in\mathcal O(a)}B(b,\rho_a/8)$, then \eqref{eq:newton-log-majorant} bounds $N_k(\cdot,a)$ on $B$, and hence
\[
 \frac1{\mu_k(B)}\int_B|N_k(\cdot,a)-(N_k(\cdot,a))_B|\,d\mu_k
 \leq 2\sup_BN_k(\cdot,a)\leq C_a.
\]
Otherwise $B\subset B(b,\rho_a/4)$ for one $b\in\mathcal O(a)$. Set
\[
 c_B=\int_{4r^2}^\infty\Gamma_k(t,z,a)\,dt.
\]
For $y\in B$, the size and spatial estimates give
\begin{align*}
 |N_k(y,a)-c_B|
 &\leq \int_0^{4r^2}\Gamma_k(t,y,a)\,dt
     +\int_{4r^2}^\infty|\Gamma_k(t,y,a)-\Gamma_k(t,z,a)|\,dt\\
 &\leq C_a\left(1+\log^+\frac{2r}{|y-b|}\right).
\end{align*}
Here $d(y,a)=|y-b|$ on $B(b,\rho_a/4)$; for the second integral,
$|y-z|<r\leq\sqrt t/2$, and
\[
 r\int_{4r^2}^{\rho_a^2}t^{-3/2}\,dt
 +r\int_{\rho_a^2}^\infty t^{-(Q+1)/2}\,dt\leq C_a.
\]
Since $w_k$ is comparable with a positive constant on this neighborhood,
\[
 \int_B\log^+\frac{2r}{|y-b|}\,d\mu_k(y)\leq C_a\mu_k(B).
\]
Therefore
\begin{align*}
 &\frac1{\mu_k(B)}
   \int_B|N_k(\cdot,a)-(N_k(\cdot,a))_B|\,d\mu_k\\
 &\qquad\leq \frac2{\mu_k(B)}
   \int_B|N_k(\cdot,a)-c_B|\,d\mu_k
 \leq C_a.
\end{align*}

It remains to consider $r\geq\rho_a/16$. Since the logarithmic term in
\eqref{eq:newton-log-majorant} is supported in the finite union
$\bigcup_{b\in\mathcal O(a)}B(b,\rho_a)$ and the weight is bounded there,
\[
 \int_{\R^2}\log^+\frac{\rho_a}{d(y,a)}\,d\mu_k(y)\leq C_a.
\]
Since $r\geq\rho_a/16$, positivity and \eqref{eq:newton-log-majorant} give
\[
 \int_BN_k(y,a)\,d\mu_k(y)\leq C_a\mu_k(B)+C_a.
\]
Hence
\[
 \frac1{\mu_k(B)}\int_B|N_k(\cdot,a)-(N_k(\cdot,a))_B|\,d\mu_k
 \leq\frac2{\mu_k(B)}\int_BN_k(\cdot,a)\,d\mu_k\leq C_a.
\]

If $a$ ranges in a compact $K\Subset\R^2_{\mathrm{reg}}$, the quantities
$\dist(K,\mathcal H)$ and $\min_{a\in K,\,\sigma\neq e}|a-\sigma(a)|$ are
positive. Thus the radius, the local weight bounds, and all constants above can be chosen
uniformly. This proves \eqref{eq:uniform-regular-poles-bmo}.
\end{proof}

As a consequence, we may prove the following pointwise bounds of the Newton kernel.

\begin{corollary}\label{cor:regular-pole-log-bounds}
Let $a\in\R^2_{\mathrm{reg}}$. There are $\rho_a,c_a,C_a>0$ such that
\[
 c_a\log\frac{\rho_a}{|x-a|}-C_a
 \leq N_k(x,a)
 \leq C_a\left(1+\log\frac{\rho_a}{|x-a|}\right)
\]
whenever $0<|x-a|<\rho_a/4$.
\end{corollary}

\begin{proof}
The upper bound is \eqref{eq:newton-log-majorant}; near $a$ the orbit distance equals $|x-a|$. Shrink $\rho_a$ so
that $\mu_k(B(a,\sqrt t))\simeq t$ for $0<t<\rho_a^2$, with comparison constants depending on $a$. Put
$r=|x-a|<\rho_a/4$. By symmetry and the Euclidean Gaussian lower bound in \eqref{eq:basic-heat-size},
\begin{align*}
 N_k(x,a)
 &\geq c_a\int_{16r^2}^{\rho_a^2}\frac1t e^{-Cr^2/t}\,dt
 \geq c_a\int_{16r^2}^{\rho_a^2}\frac{dt}{t}\\
 &=2c_a\log\frac{\rho_a}{4r}
 \geq c_a'\log\frac{\rho_a}{r}-C_a.
\end{align*}
In particular, $N_k(x,a)\to\infty$ as $x\to a$, with logarithmic order.
\end{proof}

\subsection{Poles on reflecting lines: the non-BMO result}

\begin{theorem}\label{thm:singular-poles-not-bmo}
If $y\in\mathcal H$, then $N_k(\cdot,y)\notin\BMO(\mu_k)$.
\end{theorem}

\begin{proof}
At the origin, \eqref{eq:newton-origin} and homogeneity give
\begin{equation}\label{eq:origin-average-growth}
 \frac1{\mu_k(B(0,r))}\int_{B(0,r)}N_k(x,0)\,d\mu_k(x)=cr^{-2\gamma}.
\end{equation}
Now let $y\neq0$ lie on the reflecting line $H$ and write $\kappa_H>0$ for its multiplicity. By
\eqref{eq:local-weight-line},
\[
 \mu_k(B(y,r))\simeq r^{2+2\kappa_H}
\]
for small $r$. If $\rho=|z-y|$ is small, then the lower estimate in \eqref{eq:basic-heat-size}, integrated only over
$\rho^2<t<2\rho^2$, gives
\[
 N_k(z,y)\geq c\int_{\rho^2}^{2\rho^2}t^{-1-\kappa_H}\,dt
 \geq c\rho^{-2\kappa_H}.
\]
Using the local form of the weight and scaling in $u=z-y$, we obtain
\begin{equation}\label{eq:line-average-growth}
 \frac1{\mu_k(B(y,r))}\int_{B(y,r)}N_k(z,y)\,d\mu_k(z)
 \geq cr^{-2\kappa_H}.
\end{equation}
For a nonnegative BMO function on a doubling space, telescoping the averages over concentric dyadic balls gives
\[
 |F_{B(y,r)}-F_{B(y,R)}|
 \leq C\left(1+\log\frac Rr\right)\|F\|_{\BMO}
 \qquad(0<r<R).
\]
The power growth obtained above contradicts this logarithmic bound.
\end{proof}

\begin{corollary}\label{cor:bmo-norm-lower-bound}
For $a\in\R^2_{\mathrm{reg}}$,
\begin{equation}\label{eq:bmo-norm-lower-bound}
 \|N_k(\cdot,a)\|_{\BMO(\mu_k)}\geq\frac c{w_k(a)}.
\end{equation}
Consequently the norm tends to infinity along every sequence of regular points $a_j$ for which
$w_k(a_j)\to0$.
\end{corollary}

\begin{proof}
Choose $r_a>0$ smaller than a fixed multiple of $\dist(a,\mathcal H)$. Then $w_k\simeq w_k(a)$ on $B(a,r_a)$.
The lower heat-kernel estimate gives
\[
 N_k(y,a)\geq\frac c{w_k(a)}\log\frac{r_a}{|y-a|},
 \qquad |y-a|<r_a/4.
\]
If $B_j=B(a,2^{-j}r_a)$ and $m_j=(N_k(\cdot,a))_{B_j}$, then
$m_j\geq cj/w_k(a)$ for all large $j$. On the other
hand, the doubling property and the BMO definition give
$|m_j-m_0|\leq Cj\|N_k(\cdot,a)\|_{\BMO(\mu_k)}$.
Divide by $j$ and let $j\to\infty$ to obtain \eqref{eq:bmo-norm-lower-bound}.
\end{proof}

\begin{corollary}\label{cor:green-newton-local-bmo}
For every $y\in B$,
\[
 G_k(\cdot,y)\in\BMO_{\mathrm{loc}}(B,\mu_k)
 \quad\Longleftrightarrow\quad
 N_k(\cdot,y)\in\BMO_{\mathrm{loc}}(B,\mu_k).
\]
\end{corollary}

\begin{proof}
The correction term in \eqref{eq:newton-green} is bounded in $B$. Subtracting a bounded function neither
creates nor removes local BMO.
\end{proof}

\begin{proof}[Proof of Theorem~\ref{thm:main-classification}]
The Newton-kernel assertions follow from \eqref{eq:newton-origin}, Theorem~\ref{thm:singular-poles-not-bmo}, and Proposition~\ref{prop:regular-poles-bmo}. The estimates \eqref{eq:origin-average-growth} and \eqref{eq:line-average-growth} show that the failures of BMO at singular poles are
local. The Green-kernel assertions now follow from Corollary~\ref{cor:green-newton-local-bmo}.
\end{proof}

\subsection{Refined estimates at reflected orbit points}

The refined factor is introduced only
now, to distinguish an actual pole from its reflected orbit points. Following \cite[(1.6)--(1.7)]{DHheat},
let $n(x,y)\in\{0,1,2\}$ be the minimal number of reflections needed to move $y$ into a closed
Weyl chamber containing $x$, and let $\Lambda(x,y,t)$ be the Dziuba\'nski--Hejna factor defined there.
In the irreducible full-rank planar case, \cite[Theorem~1.2]{DHheat} gives
an equivalent two-sided heat-kernel estimate with a simpler three-case
factor $\Lambda_D$.  For a rank-one root system only the cases
$n(x,y)=0,1$ occur.  We continue to use the general factor $\Lambda$ below.
We need only the following
direct consequence of its definition. If $n(x,y)\geq1$, every admissible sequence is nonempty,
so every summand contains the factor $(1+|x-y|/\sqrt t)^{-2}$. Consequently, if $n(x,y)\geq1$ and
$|x-y|\geq c_0>0$, then for $0<t\leq1$,
\begin{equation}\label{eq:lambda-reflected-bound}
 \Lambda(x,y,t)\leq C_{c_0}t.
\end{equation}
We shall use the following sharp two-sided estimates of Dziuba\'nski and Hejna \cite[Theorem~1.1]{DHheat}.

\begin{theorem}\label{thm:refined-heat-kernel}
For every $0<c_u<1/4<c_l<\infty$, there are $C_u,C_l>0$ such that
\begin{equation}\label{eq:refined-heat-kernel}
\begin{aligned}
 \frac{C_l e^{-c_ld(x,y)^2/t}\Lambda(x,y,t)}
      {\mu_k(B(x,\sqrt t))}
 &\leq\Gamma_k(t,x,y),\\
 \Gamma_k(t,x,y)
 &\leq \frac{C_u e^{-c_ud(x,y)^2/t}\Lambda(x,y,t)}
      {\mu_k(B(x,\sqrt t))}.
\end{aligned}
\end{equation}
for all $x,y\in\R^2$ and $t>0$.
\end{theorem}

This is \cite[Theorem~1.1]{DHheat} in the present notation. Notice that neither this theorem nor $\Lambda$ was
used in Proposition~\ref{prop:regular-poles-bmo}.

We first use the refined factor to rule out singular behavior at reflected orbit points distinct
from the actual pole.

\begin{lemma}\label{lem:reflected-orbit-boundedness}
Let $a\in\R^2_{\mathrm{reg}}$. There are $\rho_a,C_a>0$ such that
\begin{equation}\label{eq:reflected-orbit-boundedness}
 \sup_{x\in B(b,\rho_a)}N_k(x,a)\leq C_a
\end{equation}
for every distinct point $b\in\mathcal O(a)\setminus\{a\}$.
\end{lemma}

\begin{proof}
Choose $0<\rho_a<1$ so that the closed balls $B(b,2\rho_a)$, $b\in\mathcal O(a)$, are pairwise disjoint
and stay a positive distance from $\mathcal H$. The contribution of $t\geq\rho_a^2$ is uniformly bounded:
\[
 \int_{\rho_a^2}^\infty\Gamma_k(t,x,a)\,dt
 \leq C_a\int_{\rho_a^2}^\infty t^{-Q/2}\,dt<\infty.
\]
For $b\in\mathcal O(a)\setminus\{a\}$, one has $|x-a|\geq c_a$ and $n(x,a)\geq1$ on $B(b,\rho_a)$. By
\eqref{eq:lambda-reflected-bound}, \eqref{eq:refined-heat-kernel}, and local volume comparability,
\[
 \Gamma_k(t,x,a)\leq C_at^{-1}\Lambda(x,a,t)\leq C_a,
 \qquad 0<t<\rho_a^2.
\]
This small-time bound is integrable and proves \eqref{eq:reflected-orbit-boundedness}.
\end{proof}

\section{Applications of Euclidean BMO}
\label{sec:applications}

\subsection{Orbit balance for atomic Newton potentials}
\label{subsec:orbit-balance}

We first pass from Euclidean balls to
the quotient geometry of the reflection group. Orbit balls see all reflected copies of a regular
pole simultaneously.

\begin{definition}\label{def:orbit-bmo}
For a Euclidean ball $B=B(x,r)$, set
\[
 \mathcal O(B)=\bigcup_{\sigma\in G}B(\sigma(x),r)=\{z:d(z,x)<r\}.
\]
For $f\in L^1_{\mathrm{loc}}(d\mu_k)$, write $f_{\mathcal O(B)}$ for its weighted average over $\mathcal O(B)$ and define
\[
 \|f\|_{\BMOG}=\sup_B\frac1{\mu_k(\mathcal O(B))}
 \int_{\mathcal O(B)}|f-f_{\mathcal O(B)}|\,d\mu_k.
\]
This is the space $\BMOG(\R^2)$ of \cite[Definition~7.2]{JiuLi}.
\end{definition}

The following inclusion is \cite[Proposition~7.4(ii)]{JiuLi}. We include the short proof to record the
constant.

\begin{lemma}\label{lem:invariant-bmo-to-orbit-bmo}
If $f$ is $G$-invariant and $f\in\BMO(\mu_k)$, then $f\in\BMOG$ and
\[
 \|f\|_{\BMOG}\leq2|G|\|f\|_{\BMO(\mu_k)}.
\]
\end{lemma}

\begin{proof}
Fix $B=B(x,r)$. Since $f$ and $\mu_k$ are $G$-invariant, the weighted averages of $f$ over the
balls $B(\sigma(x),r)$ all equal $f_B$. Hence
\begin{align*}
 \inf_c\int_{\mathcal O(B)}|f-c|\,d\mu_k
 &\leq\int_{\mathcal O(B)}|f-f_B|\,d\mu_k\\
 &\leq\sum_{\sigma\in G}\int_{B(\sigma(x),r)}|f-f_{B(\sigma(x),r)}|\,d\mu_k\\
 &\leq|G|\mu_k(B)\|f\|_{\BMO(\mu_k)}.
\end{align*}
Since $B\subset\mathcal O(B)$, one has $\mu_k(B)\leq\mu_k(\mathcal O(B))$. The inequality
\[
 \int_E|f-f_E|\,d\mu_k\leq2\inf_c\int_E|f-c|\,d\mu_k
\]
now proves the assertion.
\end{proof}

For a regular point $a$, define the unnormalized orbit sum
\[
 N_k^G(x,a)=\sum_{b\in\mathcal O(a)}N_k(x,b).
\]
The stabilizer of a regular point is trivial, so this is the group sum in \cite[Remark~5.10, (5.41)]{GLR}.
With the full-space invariant projection convention one instead uses $|G|^{-1}N_k^G$. See
\cite[Section~3.2]{GraczykSawyer} for sharp estimates in type $A$.

The next result shows that orbit BMO detects precisely whether the coefficients
of an atomic source supported on a regular orbit are constant along the orbit.

\begin{theorem}\label{thm:orbit-balance}
Let $a\in\R^2_{\mathrm{reg}}$, let $(c_b)_{b\in\mathcal O(a)}$ be real numbers, and set
\[
 U_c(x)=\sum_{b\in\mathcal O(a)}c_bN_k(x,b).
\]
Then $U_c\in\BMO(\mu_k)$, and
\[
 U_c\in\BMOG
 \quad\Longleftrightarrow\quad
 c_b\text{ is independent of }b\in\mathcal O(a).
\]
In particular,
\[
 N_k(\cdot,a)\in\BMO(\mu_k)\setminus\BMOG,
 \qquad N_k^G(\cdot,a)\in\BMOG.
\]
Moreover, for every compact set $K\Subset\R^2_{\mathrm{reg}}$,
\[
 \sup_{a\in K}\|N_k^G(\cdot,a)\|_{\BMOG}<\infty.
\]
\end{theorem}

\begin{proof}
Because $\gamma>0$ and $a$ is regular, $\mathcal O(a)$ contains at least two points. Indeed, reflection in
any root hyperplane moves $a$. Every point of $\mathcal O(a)$ is regular, so Proposition~\ref{prop:regular-poles-bmo} shows that
each summand belongs to $\BMO(\mu_k)$. Hence $U_c\in\BMO(\mu_k)$.

If all coefficients have the same value, then the covariance of $N_k$ shows that $U_c$ is $G$-invariant.
The asserted $\BMOG$ membership follows from Lemma~\ref{lem:invariant-bmo-to-orbit-bmo}.

Conversely, suppose that $c_b\neq c_{b'}$ for two points $b,b'\in\mathcal O(a)$. Choose $\rho>0$ so that the balls
$B(q,2\rho)$, $q\in\mathcal O(a)$, are pairwise disjoint and so that Corollary~\ref{cor:regular-pole-log-bounds} and Lemma~\ref{lem:reflected-orbit-boundedness} apply at
every orbit point. For $0<r<\rho$, put
\[
 A_r=\frac1{\mu_k(B(q,r))}\int_{B(q,r)}N_k(x,q)\,d\mu_k(x).
\]
By the covariance of $N_k$ and the invariance of $\mu_k$, this quantity is independent of $q\in\mathcal O(a)$.
Corollary~\ref{cor:regular-pole-log-bounds} gives
\[
 A_r\geq c\log\frac1r-C,
\]
and hence $A_r\to\infty$ as $r\downarrow0$.

If $p,q\in\mathcal O(a)$ and $p\neq q$, then Lemma~\ref{lem:reflected-orbit-boundedness}, applied with pole $q$, bounds $N_k(\cdot,q)$ uniformly
on $B(p,\rho)$. Since the orbit is finite, it follows that
\[
 (U_c)_{B(q,r)}=c_qA_r+O_{a,c}(1),
 \qquad q\in\mathcal O(a),
\]
uniformly for $0<r<\rho$. Therefore
\[
 \bigl|(U_c)_{B(b,r)}-(U_c)_{B(b',r)}\bigr|
 \geq|c_b-c_{b'}|A_r-O_{a,c}(1)\longrightarrow\infty.
\]
Set $E_r=\mathcal O(B(a,r))$. For small $r$, this is the disjoint union of the balls $B(q,r)$, $q\in\mathcal O(a)$,
and all these balls have the same measure. For every constant $c$, Jensen's inequality and the
triangle inequality give
\begin{align*}
 &\int_{B(b,r)}|U_c-c|\,d\mu_k
 +\int_{B(b',r)}|U_c-c|\,d\mu_k\\
 &\hspace{35mm}\geq\mu_k(B(a,r))
 \bigl|(U_c)_{B(b,r)}-(U_c)_{B(b',r)}\bigr|.
\end{align*}
Taking $c=(U_c)_{E_r}$ and dividing by $\mu_k(E_r)=|\mathcal O(a)|\mu_k(B(a,r))$ shows that the mean oscillation
on $E_r$ tends to infinity. Thus $U_c\notin\BMOG$.

Finally, if $a$ ranges over $K\Subset\R^2_{\mathrm{reg}}$, then $\bigcup_{\sigma\in G}\sigma K$ is compact in $\R^2_{\mathrm{reg}}$. The uniform Euclidean-
BMO estimate \eqref{eq:uniform-regular-poles-bmo}, followed by Lemma~\ref{lem:invariant-bmo-to-orbit-bmo}, gives the final assertion.
\end{proof}
\subsection{Off-diagonal regularity in a chamber}

The following estimate is obtained by integrating \eqref{eq:basic-heat-regularity}.

\begin{proposition}\label{prop:off-diagonal-regularity}
Let $D$ be an open Weyl chamber and let $K\Subset U\Subset D$. There is $C_{K,U}>0$ such
that, for $x,y,y'\in K$, $x\neq y$, and $2|y-y'|\leq|x-y|$,
\begin{equation}\label{eq:newton-regularity-y}
 |N_k(x,y')-N_k(x,y)|\leq C_{K,U}\frac{|y-y'|}{|x-y|}.
\end{equation}
The same estimate holds in the first variable for $x,x',y\in K$, $x\neq y$:
\begin{equation}\label{eq:newton-regularity-x}
 |N_k(x',y)-N_k(x,y)|\leq C_{K,U}\frac{|x'-x|}{|x-y|}
\end{equation}
whenever $2|x'-x|\leq|x-y|$.
\end{proposition}

\begin{proof}
Put $r=|x-y|$, $s=|y-y'|$, and $\delta=\dist(U,\mathcal H)$. Points in the same open Weyl chamber
satisfy $d(x,y)=|x-y|$ and $d(x,y')\geq r/2$. Split the heat integral at $4s^2$. On $0<t<4s^2$, use
the two size estimates; on $t\geq4s^2$, the condition $s\leq\sqrt t/2$ permits the use of
\eqref{eq:basic-heat-regularity}. If $r<\delta/4$, local volume comparability gives
\[
 \int_0^{4s^2}\bigl(\Gamma_k(t,x,y)+\Gamma_k(t,x,y')\bigr)\,dt\leq C\frac sr,
\]
\[
 \int_{4s^2}^{\delta^2}|\Gamma_k(t,x,y)-\Gamma_k(t,x,y')|\,dt
 \leq Cs\int_{4s^2}^{\delta^2}t^{-3/2}e^{-cr^2/t}\,dt
 \leq C\frac sr.
\]
For $t\geq\delta^2$, the lower volume bound gives an additional term bounded by $C_{K,U}s$, which is
at most $C_{K,U}s/r$ because $r$ is bounded on $K$.

Suppose now that $r\geq\delta/4$.  Then
$d(x,y)=r\geq\delta/4$ and $d(x,y')\geq r/2\geq\delta/8$.  If
$s<\delta/16$, the same splitting argument, using Gaussian decay below
$\delta^2$ and the regularity estimate above $4s^2$, gives a bound by
$C_{K,U}s$.  If $s\geq\delta/16$, integrating the size estimate gives a
uniform bound for both Newton kernels on this compact off-orbit region, and this bound is
at most $C_{K,U}s/r$ because $r$ is bounded on $K$.  Thus in both cases the
difference is bounded by $C_{K,U}s/r$.  This proves
\eqref{eq:newton-regularity-y}; symmetry gives
\eqref{eq:newton-regularity-x}.
\end{proof}

\subsection{Local atomic Hardy results}
\label{subsec:local-atomic-hardy}

The atoms below are localized versions of the Coifman--
Weiss atoms \cite{CoifmanWeiss}. For the global atomic characterization of the Hardy space associated with the
Dunkl Laplacian, see \cite{DHatomic}.

For a bounded open set $\Omega\subset\R^2$, write
\[
 \delta_\Omega:=\dist(\overline\Omega,\mathcal H).
\]
Thus $\delta_\Omega>0$ means precisely that $\Omega$ is uniformly separated from all reflecting lines.

\begin{definition}\label{def:local-atoms}
Let $\Omega\subset\R^2$ be open. A local $L^\infty$-atom in $\Omega$ is a function $a$ for which there
is a Euclidean ball $B\Subset\Omega$ such that
\[
 \supp a\subset B,
 \qquad \int_Ba\,d\mu_k=0,
 \qquad \|a\|_{L^\infty(d\mu_k)}\leq\mu_k(B)^{-1}.
\]
The space $H^1_{\mathrm{at}}(\Omega,\mu_k)$ consists of all sums
\[
 f=\sum_j\lambda_ja_j,
 \qquad \sum_j|\lambda_j|<\infty,
\]
where $a_j$ are local atoms in $\Omega$ and the series converges in $L^1(\Omega,d\mu_k)$. The space is equipped
with the usual atomic norm.
\end{definition}

\subsubsection{Atomic potential estimate}

For $f$ supported in $\Omega$, define formally
\[
 N_\Omega f(x)=\int_\Omega N_k(x,y)f(y)\,d\mu_k(y),
 \qquad x\in\Omega.
\]
Under the separation hypothesis below, the expression for atoms and $H^1_{\mathrm{at}}$-functions is
understood through the $H^1$--BMO pairing. In the necessity argument, ordinary atom potentials are
evaluated only at regular points.

\begin{theorem}\label{thm:atomic-newton-potential}
Let $\Omega\subset\R^2$ be bounded and open, and assume
\[
 \delta_\Omega=\dist(\overline\Omega,\mathcal H)>0.
\]
Then $N_\Omega$ extends to a bounded operator
\[
 N_\Omega:H^1_{\mathrm{at}}(\Omega,\mu_k)\longrightarrow L^\infty(\Omega)
\]
with
\[
 \|N_\Omega f\|_{L^\infty(\Omega)}
 \leq C_\Omega\|f\|_{H^1_{\mathrm{at}}(\Omega,\mu_k)}.
\]
Moreover, for every $f\in H^1_{\mathrm{at}}(\Omega,\mu_k)$ and every $x\in\Omega$,
\[
 N_\Omega f(x)=\langle N_k(x,\cdot),f\rangle_{H^1\text{--}\mathrm{BMO}},
\]
and the definition is independent of the chosen atomic decomposition.
\end{theorem}

\begin{proof}
Since $\overline\Omega$ is compact and disjoint from $\mathcal H$, \eqref{eq:uniform-regular-poles-bmo} and symmetry give
\[
 M_\Omega:=\sup_{x\in\overline\Omega}\|N_k(x,\cdot)\|_{\BMO(\mu_k)}<\infty.
\]
After extension by zero, every local atom is a global Coifman--Weiss
$(1,\infty)$-atom on $(\mathbb R^2,|\cdot|,\mu_k)$.  Hence
$H^1_{\mathrm{at}}(\Omega,\mu_k)$ embeds continuously into the global
Coifman--Weiss atomic Hardy space, with norm at most the local atomic norm.
In particular, the $H^1$--BMO pairing below is well defined and depends only
on the function, not on its local atomic representation.
Let $a$ be a local atom supported in $B\Subset\Omega$. For fixed $x\in\Omega$, define
\[
 N_\Omega a(x):=\langle N_k(x,\cdot),a\rangle_{H^1\text{--}\mathrm{BMO}}.
\]
The cancellation of $a$ gives the concrete formula
\[
 N_\Omega a(x)=\int_B
 \bigl(N_k(x,y)-(N_k(x,\cdot))_B\bigr)a(y)\,d\mu_k(y).
\]
Consequently,
\begin{align*}
 |N_\Omega a(x)|
 &\leq\mu_k(B)^{-1}\int_B
 \bigl|N_k(x,y)-(N_k(x,\cdot))_B\bigr|\,d\mu_k(y)\\
 &\leq\|N_k(x,\cdot)\|_{\BMO(\mu_k)}\leq M_\Omega.
\end{align*}
Taking the supremum over $x\in\Omega$ gives
\[
 \|N_\Omega a\|_{L^\infty(\Omega)}\leq M_\Omega
\]
for every local atom.

If $f=\sum_{j=1}^N\lambda_ja_j$ is a finite atomic sum, set
$N_\Omega f=\sum_{j=1}^N\lambda_jN_\Omega a_j$. Then
\[
 \|N_\Omega f\|_{L^\infty(\Omega)}\leq M_\Omega\sum_{j=1}^N|\lambda_j|.
\]
For a general atomic representation $f=\sum_j\lambda_ja_j$, the series
\[
 \sum_j\lambda_jN_\Omega a_j
\]
converges absolutely in $L^\infty(\Omega)$. This defines $N_\Omega f$ and gives the asserted bound after taking
the infimum over all atomic decompositions. For each fixed $x$, the same calculation says that
the value equals the $H^1$--BMO pairing with $N_k(x,\cdot)$; hence it does not depend on the chosen
decomposition.
\end{proof}

\begin{corollary}\label{cor:continuous-atomic-potentials}
Assume $\Omega\subset\R^2$ is bounded and $\delta_\Omega>0$. Then $N_\Omega a$ has a continuous representative on $\Omega$ for every local atom $a$. If $f=\sum_j\lambda_ja_j$ with $\sum_j|\lambda_j|<\infty$, then $N_\Omega f$ is the
uniform limit on $\Omega$ of continuous finite atomic potentials.
\end{corollary}

\begin{proof}
Let $a$ be supported in $B\Subset\Omega$. The representative constructed above is
\[
 N_\Omega a(x)=\int_BN_k(x,y)a(y)\,d\mu_k(y),
\]
and the logarithmic local singularity makes the integral absolutely convergent. Near $B$, split the
integral into a small neighborhood of the diagonal and its complement. Uniform integrability
of the logarithmic majorant controls the first part.  The explicit heat-kernel formula shows that
$\Gamma_k(t,x,y)$ is jointly continuous in the spatial variables, and the basic size estimate gives
an integrable majorant on compact sets where the orbit distance is positive.  This gives continuity
of the second part and also proves continuity away from $B$ unless a distinct reflected orbit point
is present.

It remains to consider a distinct reflected orbit point, which is not covered
by Proposition~\ref{prop:off-diagonal-regularity}.  Fix
$x_0\in\Omega\setminus\overline B$ and $y_0\in\overline B$ with
$d(x_0,y_0)=0$ and $x_0\neq y_0$.  Since $\overline\Omega$ is separated from
$\mathcal H$, there are neighborhoods $V\ni x_0$ and $W\ni y_0$, contained
in fixed Weyl chambers, and constants $c_0>0$ and $0<t_0\leq1$ such that
$|x-y|\geq c_0$ and $n(x,y)\geq1$ on $V\times W$.  Local volume
comparability, \eqref{eq:lambda-reflected-bound}, and
\eqref{eq:refined-heat-kernel} give
$\Gamma_k(t,x,y)\leq C$ on $V\times W$ for $0<t\leq t_0$.  For $t\geq t_0$,
\eqref{eq:basic-heat-size} and \eqref{eq:volume-estimate} give
$\Gamma_k(t,x,y)\leq Ct^{-Q/2}$.  Since the heat kernel is jointly
continuous and $Q>2$, dominated convergence shows
that $N_k$ is jointly continuous on $V\times W$.  The orbit of $x_0$ is
finite, so the same argument at the finitely many points of
$\mathcal O(x_0)\cap\overline B$ justifies dominated convergence in the
$y$-integral.  This proves continuity at every distinct reflected orbit point.

The assertion for $f=\sum_j\lambda_ja_j$ follows from the uniform convergence in the preceding
theorem.
\end{proof}

\subsubsection{H\"older regularity of local potentials}

\begin{corollary}\label{cor:holder-local-potentials}
Let $D$ be an open Weyl chamber. Let $\Omega\subset D$ be bounded and open
with
\[
 \delta_\Omega=\dist(\overline\Omega,\mathcal H)>0.
\]
Let $1<p\leq\infty$. If $f\in L^p(\Omega,d\mu_k)$, then $N_\Omega f$ has a H\"older continuous representative on $\overline\Omega$.
More precisely,
\[
 \|N_\Omega f\|_{C^\alpha(\overline\Omega)}
 \leq C_{\Omega,p,\alpha}\|f\|_{L^p(\Omega,d\mu_k)}
\]
for every
\[
 0<\alpha<\min\left\{1,2-\frac2p\right\}.
\]
\end{corollary}

\begin{proof}
Because $\delta_\Omega>0$, the measure $d\mu_k=w_k\,dx$ is comparable with Lebesgue measure on $\overline\Omega$.
Choose an open set $U$ such that
$\overline\Omega\Subset U\Subset D$.  For $f\in L^p(\Omega,d\mu_k)$ and
$x\in\overline\Omega$, define
\[
 N_\Omega f(x)=\int_\Omega N_k(x,y)f(y)\,d\mu_k(y).
\]
This integral is absolutely convergent, uniformly in
$x\in\overline\Omega$: near $y=x$ the compact-uniform regular-pole bound is
logarithmic, and away from $x$ the kernel is bounded on compact subsets of
the chamber.  In particular,
\[
 \sup_{x\in\overline\Omega}|N_\Omega f(x)|
 \leq C_{\Omega,p}\|f\|_{L^p(\Omega,d\mu_k)}.
\]

Let $x,x'\in\overline\Omega$, set $h=|x-x'|$, and split
\[
 \Omega=E_1\cup E_2,
 \qquad E_1=B(x,4h)\cup B(x',4h),
 \qquad E_2=\Omega\setminus E_1.
\]
Let $p'$ be the conjugate exponent, with the convention $p'=1$ when $p=\infty$. If $h$ is bounded
below by a fixed positive number, the estimate follows from the $L^p\to L^\infty$ bound obtained
from the logarithmic kernel estimate. Hence assume $h$ is small. Then
\[
 N_\Omega f(x)-N_\Omega f(x')
 =\int_\Omega\bigl(N_k(x,y)-N_k(x',y)\bigr)f(y)\,d\mu_k(y).
\]
On $E_1$, the logarithmic regular-pole bound and H\"older's inequality give
\[
 \int_{E_1}\bigl(|N_k(x,y)|+|N_k(x',y)|\bigr)|f(y)|\,d\mu_k(y)
 \leq Ch^{2/p'}(1+|\log h|)\|f\|_{L^p}.
\]
This is $O(h^\alpha)\|f\|_{L^p}$ for every $\alpha<2/p'=2-2/p$.

On $E_2$, Proposition~\ref{prop:off-diagonal-regularity}, applied with
$K=\overline\Omega$ and the set $U$ chosen above, gives
\[
 |N_k(x,y)-N_k(x',y)|\leq C\frac{|x-x'|}{|x-y|}
\]
whenever $2|x-x'|\leq|x-y|$. Summing over dyadic annuli gives
\[
 \int_{E_2}|N_k(x,y)-N_k(x',y)|\,|f(y)|\,d\mu_k(y)
 \leq Ch^\alpha\|f\|_{L^p}
\]
for every $\alpha<\min\{1,2/p'\}$. Indeed, on the annulus
$2^jh\leq|x-y|<2^{j+1}h$, the kernel difference
contributes $2^{-j}$, while H\"older's inequality contributes $(2^jh)^{2/p'}$; summing the resulting geometric series yields $h^\alpha$ for every $\alpha<\min\{1,2/p'\}$. Combining the near and far estimates proves
the claim on $\overline\Omega$.
\end{proof}

\subsubsection{Sharpness and necessity}

The following zero-mean radial test captures the blow-up as a
regular evaluation point approaches a reflecting line.

\begin{lemma}\label{lem:radial-test}
There exist a radial function $q\in C_c^\infty(B(0,1))$, with
\[
 \int_{\R^2}q(z)\,dz=0,
\]
and constants $c,c_0>0$, depending only on $R$ and $k$, such that
\begin{equation}\label{eq:radial-test-lower}
 \int_{B(0,1)}N_k(p,p+rz)q(z)\,dz\geq\frac c{w_k(p)}
\end{equation}
whenever $p\in\R^2_{\mathrm{reg}}$ and $0<r<c_0\dist(p,\mathcal H)$.
\end{lemma}

\begin{proof}
Fix a large number $L$. Choose $h\in C_c^\infty([0,\infty))$ such that $h(s)=s$ near zero, $h'\geq0$ on
$(0,e^{-L})$, $h$ is a positive constant on $[e^{-L},1/2]$, $h'\leq0$ on $(1/2,3/4)$, and $h=0$ on $[3/4,\infty)$.
The transitions may be chosen flat at their endpoints. For $z\neq0$, put
\[
 q(z)=\frac{h(|z|)h'(|z|)}{|z|},
\]
and set $q(0)=1$. Since $h(s)=s$ near zero, $q$ is smooth at the origin; the flat transitions make
$q\in C_c^\infty(B(0,1))$. Thus $q\geq0$ on $|z|<e^{-L}$, $q\leq0$ on $1/2<|z|<3/4$, and $q=0$ elsewhere. If
$H$ is the value of $h$ on $[e^{-L},1/2]$, then
\[
 \int_{\{q>0\}}q\,dz=-\int_{\{q<0\}}q\,dz=\pi H^2=:M_L,
 \qquad \int q\,dz=0.
\]

Write $w_p=w_k(p)$ and fix $A_0>16$. The root system is finite, so after decreasing $c_0$, the
ball $B(p,\sqrt{A_0}r)$ stays in the open Weyl chamber containing $p$. In one closed chamber the identity
minimizes orbit distance. The volume formula therefore gives
\[
 \mu_k(B(p,\sqrt t))\simeq w_pt,
 \qquad d(p,y)=|p-y|
\]
for $0<t<A_0r^2$ and $y\in B(p,r)$. Hence the basic heat-kernel estimates yield
\begin{equation}\label{eq:local-two-sided-heat}
 \frac c{w_pt}e^{-C|p-y|^2/t}
 \leq\Gamma_k(t,p,y)
 \leq\frac C{w_pt}e^{-c|p-y|^2/t}
\end{equation}
in this range. Decompose
\[
 N_k(p,y)=S(y)+R(y),
 \qquad S(y)=\int_0^{A_0r^2}\Gamma_k(t,p,y)\,dt.
\]
The lower estimate in \eqref{eq:local-two-sided-heat} gives
\[
 S(p+rz)\geq\frac{cL-C}{w_p}
 \qquad(|z|<e^{-L}),
\]
whereas the upper estimate gives
\[
 S(p+rz)\leq\frac C{w_p}
 \qquad(1/2<|z|<3/4).
\]
For the contribution $R$ of $t\geq A_0r^2$, \eqref{eq:basic-heat-regularity} and the global consequence
$\mu_k(B(p,\sqrt t))\gtrsim w_pt$ of \eqref{eq:volume-estimate} give
\[
 |R(p+rz)-R(p)|
 \leq Cr|z|\int_{A_0r^2}^\infty\frac{dt}{w_pt^{3/2}}
 \leq\frac C{w_p},
 \qquad |z|<1.
\]
Using $\int q\,dz=0$ to remove the constant part of $R$, and absorbing the negative short-time
contribution and the large-time oscillation into $C$, we obtain
\[
 \int_{B(0,1)}N_k(p,p+rz)q(z)\,dz
 \geq\frac{M_L}{w_p}(cL-C).
\]
Choose $L$ once and for all so that the last bracket is positive. This proves \eqref{eq:radial-test-lower}.
\end{proof}

\begin{theorem}\label{thm:atomic-necessity}
Let $\Omega\subset\R^2$ be bounded and open. If
\[
 \delta_\Omega=\dist(\overline\Omega,\mathcal H)=0,
\]
then there are local atoms $a_j$ in $\Omega$ such that
\begin{equation}\label{eq:atomic-blowup}
 \|a_j\|_{H^1_{\mathrm{at}}(\Omega,\mu_k)}\leq1,
 \qquad \|N_\Omega a_j\|_{L^\infty(\Omega)}\longrightarrow\infty.
\end{equation}
Consequently, there is no bounded extension $H^1_{\mathrm{at}}(\Omega,\mu_k)\to L^\infty(\Omega)$ that agrees with the Newton
potentials of local atoms at regular evaluation points.
\end{theorem}

\begin{proof}
Choose $p_j\in\Omega\setminus\mathcal H$ such that
\[
 \delta_j:=\dist(p_j,\mathcal H)\longrightarrow0.
\]
After passing to a subsequence, $p_j\to p_*\in\overline\Omega\cap\mathcal H$, and therefore $w_k(p_j)\to0$. Put
$d_j=\dist(p_j,\partial\Omega)$ and choose
\[
 0<r_j<c_0\delta_j,
 \qquad 2r_j<d_j,
\]
where $c_0$ is as in Lemma~\ref{lem:radial-test}. Define
\[
 q_j(y)=r_j^{-2}q\!\left(\frac{y-p_j}{r_j}\right),
\]
where $q$ is the function in Lemma~\ref{lem:radial-test}. Then
$\int q_j\,dy=0$, and
\[
 \supp q_j\subset B(p_j,r_j)\Subset\Omega.
\]
On $B(p_j,r_j)$, the weight is comparable with $w_k(p_j)$ and
$\mu_k(B(p_j,r_j))\simeq w_k(p_j)r_j^2$. Hence a
sufficiently small constant $c_*>0$, independent of $j$, makes
\[
 a_j(y)=
 \begin{cases}
 \displaystyle c_*\frac{q_j(y)}{w_k(y)},&y\in B(p_j,r_j),\\[6pt]
 0,&y\notin B(p_j,r_j),
 \end{cases}
\]
a local atom: its weighted integral vanishes because $\int q_j\,dy=0$, and its $L^\infty$ norm is at most
$\mu_k(B(p_j,r_j))^{-1}$. Changing variables and using Lemma~\ref{lem:radial-test}, we get
\begin{align*}
 N_\Omega a_j(p_j)
 &=c_*\int_\Omega N_k(p_j,y)q_j(y)\,dy\\
 &=c_*\int_{B(0,1)}N_k(p_j,p_j+r_jz)q(z)\,dz
 \geq\frac c{w_k(p_j)}.
\end{align*}
For each fixed $j$, this potential is continuous near $p_j$: the source is bounded and supported in
a ball separated from $\mathcal H$, and the local kernel singularity is logarithmic. Hence its essential
supremum is at least its value at $p_j$. Since $w_k(p_j)\to0$, \eqref{eq:atomic-blowup} follows.
\end{proof}

The preceding sufficiency and necessity results give the exact geometric criterion.

\begin{corollary}\label{cor:atomic-geometric-criterion}
For every bounded open set $\Omega\subset\R^2$, the natural atom-level Newton map
admits a bounded extension
\[
 H^1_{\mathrm{at}}(\Omega,\mu_k)\longrightarrow L^\infty(\Omega)
\]
if and only if $\dist(\overline\Omega,\mathcal H)>0$.
\end{corollary}
\section{A Wente-type Newton-potential estimate}\label{sec:wente}

The first result in this section is an intrinsic Dunkl version of the
Coifman--Lions--Meyer--Semmes Jacobian estimate, which is then applied to the
global Dunkl Newton kernel.  We subsequently consider the Euclidean Jacobian.
For data supported in an open Weyl chamber, a change of unknown identifies the two
source measures, and the corresponding energies are related by an exact
identity.  All functions in the Jacobian statements below are real-valued.

\subsection{The Dunkl Hardy space and its dual}

For $f\in L^1(\mathbb R^2,d\mu_k)$, put
\[
  \mathcal M_{\Delta_k}f(x)
  =\sup_{t>0}\ \sup_{|x-y|<t}
       \big|e^{t^2\Delta_k}f(y)\big|.
\]
The Dunkl Hardy space is
\[
  H_k^1(\mathbb R^2)=H_{\Delta_k}^1(\mathbb R^2)
  =\big\{f\in L^1(d\mu_k):
       \mathcal M_{\Delta_k}f\in L^1(d\mu_k)\big\},
\]
with norm $\|f\|_{H_k^1}=\|\mathcal M_{\Delta_k}f\|_{L^1(d\mu_k)}$.
The theorem of Dziuba\'nski and Hejna \cite[Theorem~1.5]{DHatomic} identifies this space
with the Coifman--Weiss atomic Hardy space, with equivalent norms, on
\[
  (\mathbb R^2,|\cdot|,\mu_k).
\]
Consequently, Coifman--Weiss duality
\cite[Theorem~B, p.~593]{CoifmanWeiss}, with $p=1$, gives an identification
with equivalent norms,
\[
  (H_k^1)^*\simeq\operatorname{BMO}_k,
  \qquad
  \operatorname{BMO}_k
  :=\operatorname{BMO}(\mu_k)/\{\text{constant functions}\},
\]
where
\[
  \|b\|_{\operatorname{BMO}_k}
  =\sup_B\frac1{\mu_k(B)}
       \int_B|b-b_{B,k}|\,d\mu_k,
  \qquad
  b_{B,k}=\frac1{\mu_k(B)}\int_B b\,d\mu_k,
\]
the supremum is over Euclidean balls, and
\[
 \|\Lambda_b\|_{(H_k^1)^*}\simeq
 \|b\|_{\operatorname{BMO}_k},
 \qquad
 \Lambda_b(f)=\int_{\mathbb R^2}f b\,d\mu_k.
\]
The pairing is initially defined on finite atomic sums and then extended
continuously.  This is not the orbit space
$\operatorname{BMO}_G$.  The latter is a proper subspace by
Theorem~\ref{thm:orbit-balance}, and it is not the dual space used below.

\subsection{A Dunkl--CLMS estimate in dimension two}

Write
\[
  L_k=-\Delta_k,
  \qquad
  R_j=T_{e_j}L_k^{-1/2},
  \qquad
  \nabla_k u=(T_{e_1}u,T_{e_2}u).
\]
For smooth functions $u,v$, define
\[
  \mathcal J_k(u,v)
  =T_{e_1}u\,T_{e_2}v-T_{e_2}u\,T_{e_1}v.
\]

Han--Lee--Li--Wick \cite{HanLeeLiWick} proved the first commutator theorem
for the Dunkl Riesz transforms.  The upper bound in that paper was stated
with the orbit BMO norm.  The later theorem of Dziuba\'nski--Hejna
\cite[(1.4) and Theorem~3.1]{DHcommutators}, applied with $p=2$ to the Dunkl
Riesz transforms, gives the Euclidean-ball estimate
\[
  \|[b,R_j]h\|_{L^2(d\mu_k)}
  \leq C_k\|b\|_{\operatorname{BMO}_k}
          \|h\|_{L^2(d\mu_k)},
  \qquad [b,R_j]=bR_j-R_jb.
\]
This later upper bound is the one needed here.

\begin{theorem}\label{thm:dunkl-CLMS}
For $u,v\in C_c^\infty(\mathbb R^2)$,
\[
  \mathcal J_k(u,v)\in H_k^1(\mathbb R^2)
\]
and
\[
  \|\mathcal J_k(u,v)\|_{H_k^1}
  \leq C_k
  \|\nabla_k u\|_{L^2(d\mu_k)}
  \|\nabla_k v\|_{L^2(d\mu_k)}.
\]
Consequently, the Jacobian map extends continuously to the completion of
$C_c^\infty(\mathbb R^2)$ in the Dunkl gradient norm.
\end{theorem}

\begin{proof}
Put
\[
  f=L_k^{1/2}u,
  \qquad
  g=L_k^{1/2}v.
\]
The Dunkl-transform multiplier identities \cite{ADH} give
\[
  T_{e_j}u=R_jf,
  \qquad
  T_{e_j}v=R_jg,
\]
and
\[
  \|f\|_{L^2(d\mu_k)}=\|\nabla_k u\|_{L^2(d\mu_k)},
  \qquad
  \|g\|_{L^2(d\mu_k)}=\|\nabla_k v\|_{L^2(d\mu_k)}.
\]
The operators $R_1,R_2$ commute and are skew-adjoint on $L^2(d\mu_k)$.
For a bounded measurable function $b$, skew-adjointness and commutation give
\begin{align*}
 \int_{\mathbb R^2}b\,\mathcal J_k(u,v)\,d\mu_k
 &=\int_{\mathbb R^2}b
       \big(R_1fR_2g-R_2fR_1g\big)\,d\mu_k\\
 &=\big\langle f,
       -R_1(bR_2g)+R_2(bR_1g)\big\rangle_{L^2(d\mu_k)}\\
 &=\big\langle f,
       [b,R_1]R_2g-[b,R_2]R_1g
     \big\rangle_{L^2(d\mu_k)}.
\end{align*}
In the last line the two terms containing $bR_1R_2g$ cancel.  The
commutator estimate and the $L^2$ boundedness of the Riesz transforms now
give
\[
 \left|\int_{\mathbb R^2}b\,\mathcal J_k(u,v)\,d\mu_k\right|
 \leq C_k\|b\|_{\operatorname{BMO}_k}
       \|\nabla_k u\|_{L^2(d\mu_k)}
       \|\nabla_k v\|_{L^2(d\mu_k)}.
\]
For general $b\in\operatorname{BMO}_k$, let
$b_N=\max\{-N,\min\{b,N\}\}$.  The truncation property of BMO gives
$\|b_N\|_{\operatorname{BMO}_k}\leq2\|b\|_{\operatorname{BMO}_k}$.
Moreover, $|b_N|\leq|b|$, while
$\mathcal J_k(u,v)\in L_c^\infty(d\mu_k)$.  Applying the preceding estimate
to $b_N$ and then using dominated convergence proves the same estimate for
$b$.

The Dunkl operators commute and are skew-symmetric with respect to
$d\mu_k$.  Hence
\begin{align*}
 \int_{\mathbb R^2}\mathcal J_k(u,v)\,d\mu_k
 &=-\int_{\mathbb R^2}u
       \big(T_{e_1}T_{e_2}-T_{e_2}T_{e_1}\big)v\,d\mu_k=0.
\end{align*}
Also, a Dunkl operator maps a compactly supported smooth function to a
smooth function supported in the finite orbit of its support.  Thus
$\mathcal J_k(u,v)$ is bounded, compactly supported, and has integral zero.
It is therefore a constant multiple of a Coifman--Weiss atom, and in
particular it belongs to $H_k^1$.  The $H_k^1$--$\operatorname{BMO}_k$
duality and the preceding estimate prove the required norm bound.

The extension follows by applying the bilinear estimate to
\begin{align*}
 &\mathcal J_k(u^{(n)},v^{(n)})-
   \mathcal J_k(u^{(m)},v^{(m)})\\
 &\quad=\mathcal J_k(u^{(n)}-u^{(m)},v^{(n)})
  +\mathcal J_k(u^{(m)},v^{(n)}-v^{(m)}).
\end{align*}
\end{proof}

\subsection{The intrinsic Newton potential}

\begin{proposition}
\label{prop:intrinsic-Newton-Wente}
Let $K\Subset\mathbb R^2_{\mathrm{reg}}$ be compact.  For
$u,v\in C_c^\infty(\mathbb R^2)$, put
\[
  \Phi_k(u,v)(x)
  =\int_{\mathbb R^2}N_k(x,y)\mathcal J_k(u,v)(y)\,d\mu_k(y).
\]
Then
\[
  \sup_{x\in K}|\Phi_k(u,v)(x)|
  \leq C_{K,k}
  \|\nabla_k u\|_{L^2(d\mu_k)}
  \|\nabla_k v\|_{L^2(d\mu_k)}.
\]
For limits in the Dunkl gradient norm, the pointwise representative on
$K$ is defined by
\[
  \Phi_k(u,v)(x)
  =\big\langle N_k(x,\cdot),\mathcal J_k(u,v)
    \big\rangle_{\operatorname{BMO}_k,H_k^1}.
\]
For smooth data this pairing is the ordinary integral above.  Moreover, for
such data,
\[
  -\Delta_k\Phi_k(u,v)=\mathcal J_k(u,v)
\]
in the $d\mu_k$-distributional sense.  Thus, for every
$\varphi\in C_c^\infty(\mathbb R^2)$,
\[
 \int_{\mathbb R^2}\Phi_k(u,v)(x)(-\Delta_k\varphi)(x)\,d\mu_k(x)
 =\int_{\mathbb R^2}\mathcal J_k(u,v)(y)\varphi(y)\,d\mu_k(y).
\]
\end{proposition}

\begin{proof}
Proposition~\ref{prop:regular-poles-bmo} and the symmetry of the Newton
kernel give
\[
  \sup_{x\in K}
  \|N_k(x,\cdot)\|_{\operatorname{BMO}_k}<\infty.
\]
The estimate follows at once from Theorem~\ref{thm:dunkl-CLMS} and
$H_k^1$--$\operatorname{BMO}_k$ duality.  For smooth compactly supported
data, the kernel is locally integrable against $d\mu_k$, so the pairing is
the displayed integral.

For the distributional identity, put
$E=\supp(\Delta_k\varphi)$, which is compact.  By symmetry, the Markov property,
the global volume lower bound, and $Q>2$,
\begin{align*}
 A_E:=\sup_{y\in\mathbb R^2}\int_E N_k(x,y)\,d\mu_k(x)
 &\leq \int_0^1\int_{\mathbb R^2}
          \Gamma_k(t,y,x)\,d\mu_k(x)\,dt\\
 &\quad+C\mu_k(E)\int_1^\infty t^{-Q/2}\,dt<\infty.
\end{align*}
Since $\mathcal J_k(u,v)$ is bounded and compactly supported, this yields
\begin{align*}
 &\int_{\mathbb R^2}|\mathcal J_k(u,v)(y)|
   \left(\int_E N_k(x,y)|\Delta_k\varphi(x)|\,d\mu_k(x)\right)
   d\mu_k(y)\\
 &\qquad\leq
 \|\Delta_k\varphi\|_\infty
 \|\mathcal J_k(u,v)\|_{L^1(d\mu_k)}A_E<\infty.
\end{align*}
Thus Fubini's theorem applies, and the fundamental-solution identity gives
\begin{align*}
 &\int_{\mathbb R^2}\Phi_k(u,v)(x)(-\Delta_k\varphi)(x)\,d\mu_k(x)\\
 &\quad=\int_{\mathbb R^2}\mathcal J_k(u,v)(y)
       \left(\int_{\mathbb R^2}N_k(x,y)
             (-\Delta_k\varphi)(x)\,d\mu_k(x)\right)d\mu_k(y)\\
 &\quad=\int_{\mathbb R^2}\mathcal J_k(u,v)(y)
             \varphi(y)\,d\mu_k(y).
\end{align*}
The pointwise definition and the estimate for gradient-norm limits follow
from convergence in $H_k^1$.
\end{proof}

\subsection{A chamber gauge identity}

\begin{lemma}\label{lem:global-chamber-gauge}
Let $D$ be an open Weyl chamber and put $\rho_k=w_k^{1/2}$.  Let
$F,G\in C_c^\infty(D)$, and define
\[
 u(x)=
 \begin{cases}
   \rho_k(x)^{-1}F(x),&x\in D,\\
   0,&x\notin D,
 \end{cases}
 \qquad
 v(x)=
 \begin{cases}
   \rho_k(x)^{-1}G(x),&x\in D,\\
   0,&x\notin D.
 \end{cases}
\]
Then $u,v\in C_c^\infty(\mathbb R^2)$ and, on $D$,
\[
 T_{e_i}u=\rho_k^{-1}F_{x_i},
 \qquad
 T_{e_i}v=\rho_k^{-1}G_{x_i},
 \qquad i=1,2.
\]
The Dunkl Jacobian vanishes almost everywhere outside $D$, and globally
\[
 \boxed{
   \mathcal J_k(u,v)\,d\mu_k
   =(F_{x_1}G_{x_2}-F_{x_2}G_{x_1})\,dx.}
\]
In particular, the two sides have the same compact support in $D$.  One
also has the exact energy identity
\[
 \|\nabla_k u\|_{L^2(d\mu_k)}^2
 =\|\nabla F\|_{L^2(dx)}^2
  +\sum_{\alpha\in R_+}k(\alpha)^2|\alpha|^2
    \int_D\frac{|F(y)|^2}{\langle\alpha,y\rangle^2}\,dy,
\]
and the analogous identity with $v$ and $G$.
\end{lemma}

\begin{proof}
Since $\supp F\cup\supp G\Subset D$ and $\rho_k$
is positive and smooth on $D$, the functions $F/\rho_k$ and $G/\rho_k$
vanish in a neighborhood of $\partial D$.  Their zero extensions therefore
belong to $C_c^\infty(\mathbb R^2)$.

If $x\in D$, then $\sigma_\alpha x\notin D$ for every
$\alpha\in R_+$.  Therefore
\begin{align*}
 T_{e_i}u(x)
 &=u_{x_i}(x)
   +u(x)\sum_{\alpha\in R_+}
       k(\alpha)\frac{\alpha_i}{\langle\alpha,x\rangle}\\
 &=\rho_k(x)^{-1}(\rho_k u)_{x_i}(x)
  =\rho_k(x)^{-1}F_{x_i}(x).
\end{align*}

Now let $x$ lie outside $D$ and outside the reflecting lines.  At most one
reflection $\sigma_\alpha$ can send the chamber containing $x$ to $D$.
Indeed, the reflection group acts simply transitively on the Weyl
chambers.  If there is no such reflection, then $\nabla_k u(x)=0$.  If
$\sigma_\alpha x\in D$, then
\[
 \nabla_k u(x)
 =-k(\alpha)\frac{u(\sigma_\alpha x)}
               {\langle\alpha,x\rangle}\,\alpha.
\]
The same formula holds for $v$.  Hence $\nabla_k u(x)$ and
$\nabla_k v(x)$ are parallel, and their determinant is zero.  Reflecting
lines have $\mu_k$-measure zero.  This proves the support claim.  On $D$,
\[
 \mathcal J_k(u,v)\,d\mu_k
 =\rho_k^{-2}(F_{x_1}G_{x_2}-F_{x_2}G_{x_1})\rho_k^2\,dx
 =(F_{x_1}G_{x_2}-F_{x_2}G_{x_1})\,dx.
\]

The contribution of $D$ to the energy of $u$ is
\[
 \int_D|\nabla_k u|^2\,d\mu_k
 =\int_D|\nabla F|^2\,dx.
\]
For $x=\sigma_\alpha y\in\sigma_\alpha D$, the preceding formula gives
\[
 |\nabla_k u(x)|^2
 =k(\alpha)^2|\alpha|^2
   \frac{|u(y)|^2}{\langle\alpha,y\rangle^2}.
\]
The chambers $\sigma_\alpha D$, $\alpha\in R_+$, are distinct, and
$d\mu_k$ is reflection invariant.  Changing variables and using
$|u|^2d\mu_k=|F|^2dx$ on $D$ proves the energy identity.
\end{proof}

\subsection{Euclidean-source corollary}

The next two results concern the Euclidean Jacobian with respect to Lebesgue
measure.  They are useful on a fixed chamber and give a direct route to the
sharp same-domain criterion below.
The classical estimate goes back to Wente \cite{Wente}.  The use of a BMO
kernel together with a Jacobian estimate is the route used by
Chanillo--Li \cite[Section~2]{ChanilloLi1992}.

We use the following theorem of Coifman, Lions, Meyer, and Semmes
\cite{CLMS}.

\begin{theorem}\label{thm:classical-CLMS}
If $F,G\in W^{1,2}(\mathbb R^2)$, then
\[
 F_{x_1}G_{x_2}-F_{x_2}G_{x_1}\in H^1(\mathbb R^2)
\]
and
\[
 \|F_{x_1}G_{x_2}-F_{x_2}G_{x_1}\|_{H^1(\mathbb R^2)}
 \leq C\|\nabla F\|_{L^2(\mathbb R^2)}
       \|\nabla G\|_{L^2(\mathbb R^2)},
\]
where $H^1(\mathbb R^2)$ is the classical real Hardy space and $C$ is
universal.
\end{theorem}

\begin{lemma}\label{lem:local-classical-CLMS}
Let $U\Subset U_1\subset\mathbb R^2$ be bounded open sets, with $U_1$
connected and Lipschitz.  Here $\operatorname{BMO}(U_1;dx)$ denotes the
intrinsic BMO space whose seminorm is taken over Euclidean balls contained
in $U_1$.  There is a constant $C_{U,U_1}$ such that, for
$F,G\in W_0^{1,2}(U)$ and $b\in\operatorname{BMO}(U_1;dx)$,
\begin{equation}\label{eq:local-classical-CLMS}
 |\langle b,F_{x_1}G_{x_2}-F_{x_2}G_{x_1}\rangle_U|
 \leq C_{U,U_1}\|b\|_{\operatorname{BMO}(U_1;dx)}
       \|\nabla F\|_{L^2}\|\nabla G\|_{L^2}.
\end{equation}
\end{lemma}

\begin{proof}
By the BMO extension theorem for bounded Lipschitz domains \cite{Jones},
$b$ has an extension $\widetilde b\in\operatorname{BMO}(\mathbb R^2)$ such
that
\[
 \|\widetilde b\|_{\operatorname{BMO}(\mathbb R^2;dx)}
 \leq C_{U_1}\|b\|_{\operatorname{BMO}(U_1;dx)}.
\]
Extend $F$ and $G$ by zero to $\mathbb R^2$ and define the left side of
\eqref{eq:local-classical-CLMS} as the global $H^1$--BMO pairing with
$\widetilde b$.  For Sobolev data this is a distributional pairing; no
absolute integrability of the product is asserted.
To see that this is independent of the extension, approximate $F,G$ in
$W_0^{1,2}(U)$ by smooth compactly supported functions.  The pairings of
the smooth Jacobians with the difference of two extensions vanish, while
Theorem~\ref{thm:classical-CLMS} passes this identity to the limit.
Indeed, the difference of the extensions vanishes on $U_1$, while the
smooth approximants are compactly supported in $U\Subset U_1$.  The pairing
agrees with the ordinary integral for smooth data.
Theorem~\ref{thm:classical-CLMS} and the Fefferman--Stein $H^1$--BMO
duality \cite{FeffermanStein} give
\eqref{eq:local-classical-CLMS}.
\end{proof}

\begin{proposition}\label{prop:Euclidean-source-Newton}
Let $D$ be an open Weyl chamber, let
$U\Subset D$ be bounded and open, and let
$K\Subset\mathbb R^2_{\mathrm{reg}}$ be compact.  For
$F,G\in C_c^\infty(U)$, define
\[
 \Psi(x)=\int_U N_k(x,y)
 \bigl(F_{y_1}(y)G_{y_2}(y)-F_{y_2}(y)G_{y_1}(y)\bigr)\,dy,
 \qquad x\in\mathbb R^2.
\]
For general $F,G\in W_0^{1,2}(U)$, define $\Psi$ as the
$L^1_{\mathrm{loc}}(\mathbb R^2,d\mu_k)$ limit under smooth $W^{1,2}$
approximation.  This limit is independent of the approximation.  On $K$,
take the pointwise representative given by the local $H^1$--BMO pairing.
Then
\begin{equation}\label{eq:Euclidean-source-bound}
 \sup_{x\in K}|\Psi(x)|
 \leq C_{K,U,k}
       \|\nabla F\|_{L^2(\mathbb R^2)}
       \|\nabla G\|_{L^2(\mathbb R^2)}.
\end{equation}
Moreover, in the global distributional sense,
\[
 -\Delta_k\Psi\,d\mu_k
 =(F_{x_1}G_{x_2}-F_{x_2}G_{x_1})\,dx.
\]
That is, for every $\varphi\in C_c^\infty(\mathbb R^2)$,
\begin{equation}\label{eq:Euclidean-source-equation}
 \int_{\mathbb R^2}\Psi(x)(-\Delta_k\varphi)(x)\,d\mu_k(x)
 =\int_U
 \bigl(F_{y_1}(y)G_{y_2}(y)-F_{y_2}(y)G_{y_1}(y)\bigr)
 \varphi(y)\,dy.
\end{equation}
\end{proposition}

\begin{proof}
First assume $F,G\in C_c^\infty(U)$.  Choose a bounded Lipschitz set
$U\Subset U_1\Subset D$.  On $U_1$, $d\mu_k$ and $dx$ are comparable.
By Proposition~\ref{prop:regular-poles-bmo} and the symmetry of $N_k$,
\begin{equation}\label{eq:local-kernel-BMO}
 \sup_{x\in K}
 \|N_k(x,\cdot)\|_{\operatorname{BMO}(U_1;dx)}<\infty.
\end{equation}
Indeed, weighted and Lebesgue BMO are equivalent on $U_1$.  Applying
Lemma~\ref{lem:local-classical-CLMS} with $b(y)=N_k(x,y)$ gives, uniformly
for $x\in K$,
\[
 |\Psi(x)|
 \leq C_{K,U,k}\|\nabla F\|_{L^2}\|\nabla G\|_{L^2},
\]
which is \eqref{eq:Euclidean-source-bound}.

For the distributional identity, rewrite
\[
 \Psi(x)=\int_U N_k(x,y)
 \frac{F_{y_1}(y)G_{y_2}(y)-F_{y_2}(y)G_{y_1}(y)}{w_k(y)}
 \,d\mu_k(y).
\]
Since $w_k^{-1}$ is smooth on $D$, the zero extension of
$(F_{x_1}G_{x_2}-F_{x_2}G_{x_1})/w_k$ belongs to
$C_c^\infty(\mathbb R^2)$.  If $E\subset\mathbb R^2$ is compact, the
regular-pole kernel
bounds give
\[
 \sup_{y\in\supp(F_{y_1}G_{y_2}-F_{y_2}G_{y_1})}
 \int_E N_k(x,y)\,d\mu_k(x)<\infty.
\]
This supplies the absolute integrability needed for Fubini's theorem.
Applying the fundamental-solution identity and Fubini's theorem gives
\eqref{eq:Euclidean-source-equation}.

For the passage to Sobolev data, choose
$F^{(n)},G^{(n)}\in C_c^\infty(U)$ converging to $F,G$ in $W^{1,2}$.
Bilinearity and Theorem~\ref{thm:classical-CLMS} show that the Jacobians
converge in
$H^1(\mathbb R^2)$; in particular, they converge in $L^1(U)$.  By
\eqref{eq:local-kernel-BMO} and
Lemma~\ref{lem:local-classical-CLMS}, the corresponding potentials converge
uniformly on $K$, independently of the approximation.

Moreover, the regular-pole kernel bounds give, for every compact
$E\subset\mathbb R^2$,
\[
 \sup_{y\in U}\int_E N_k(x,y)\,d\mu_k(x)<\infty.
\]
Consequently the potentials converge in $L^1(E,d\mu_k)$ as well.  The two
limits agree on compact regular sets.  We may therefore pass to the limit
in \eqref{eq:Euclidean-source-equation}, which proves the distributional
identity for Sobolev data.
An exhaustion of $\mathbb R^2_{\mathrm{reg}}$ by compact sets
shows that the pointwise representatives obtained from different choices
of $K$ agree.
\end{proof}

The radial test used for the atomic obstruction has a compactly supported
Jacobian realization.  This is the only additional ingredient needed for
the negative result.

\begin{lemma}\label{lem:Jacobian-realization}
For the function $q$ fixed in Lemma~\ref{lem:radial-test}, there are
$U,V\in C_c^\infty(B(0,1))$ such that
\[
 U_{z_1}V_{z_2}-U_{z_2}V_{z_1}=q
\]
and
\[
 \int_{\mathbb R^2}\big(|\nabla U|^2+|\nabla V|^2\big)\,dz<\infty.
\]
\end{lemma}

\begin{proof}
Let $h$ be the profile used in the proof of Lemma~\ref{lem:radial-test}.
For $z\neq0$, set
\[
 U(z)=h(|z|)\frac{z_1}{|z|},
 \qquad
 V(z)=h(|z|)\frac{z_2}{|z|},
\]
and set $U(0)=V(0)=0$.  Since $h(s)=s$ near zero and its transitions are
flat, $U,V\in C_c^\infty(B(0,1))$.  A direct calculation, with $s=|z|$,
gives
\[
 U_{z_1}(z)V_{z_2}(z)-U_{z_2}(z)V_{z_1}(z)
 =\frac{h(s)h'(s)}s=q(z).
\]
Moreover,
\[
 \int_{\mathbb R^2}\big(|\nabla U|^2+|\nabla V|^2\big)\,dz
 =2\pi\int_0^\infty
   \left(h'(s)^2+\frac{h(s)^2}{s^2}\right)s\,ds<\infty.
\]
\end{proof}

\subsection{The intrinsic obstruction at the reflecting lines}

\begin{proposition}\label{prop:intrinsic-Wente-obstruction}
Let $\Omega\subset\mathbb R^2$ be a bounded domain and assume
\[
 \dist(\overline\Omega,\mathcal H)=0.
\]
Then there are
\[
 p_*\in\overline\Omega\cap\mathcal H,
 \qquad
 p_j\in\Omega\cap\mathbb R^2_{\mathrm{reg}},
 \qquad p_j\longrightarrow p_*,
\]
radii $r_j>0$, and functions
$u^{(j)},v^{(j)}\in C_c^\infty(B(p_j,r_j))$ such that
\[
 \overline{B(p_j,r_j)}
 \subset\Omega\cap\mathbb R^2_{\mathrm{reg}},
 \qquad
 \sup_j\big(
   \|\nabla_k u^{(j)}\|_{L^2(d\mu_k)}
  +\|\nabla_k v^{(j)}\|_{L^2(d\mu_k)}
 \big)<\infty,
\]
while
\[
 \int_\Omega N_k(p_j,y)
 \mathcal J_k(u^{(j)},v^{(j)})(y)\,d\mu_k(y)
 \geq\frac{c}{w_k(p_j)}\longrightarrow\infty.
\]
In particular, there is no constant $C_{\Omega,k}<\infty$ such that
\begin{equation}\label{eq:intrinsic-same-domain-failure}
 \left\|
  \int_\Omega N_k(\cdot,y)\mathcal J_k(u,v)(y)\,d\mu_k(y)
 \right\|_{L^\infty(\Omega)}
 \leq C_{\Omega,k}
 \|\nabla_k u\|_{L^2(d\mu_k)}
 \|\nabla_k v\|_{L^2(d\mu_k)}
\end{equation}
for every $u,v\in C_c^\infty(\Omega\setminus\mathcal H)$.
\end{proposition}

\begin{proof}
Choose $p_j\in\Omega\cap\mathbb R^2_{\mathrm{reg}}$ such that
\[
 \delta_j=\dist(p_j,\mathcal H)\longrightarrow0.
\]
After passing to a subsequence,
$p_j\to p_*\in\overline\Omega\cap\mathcal H$, and hence
$w_k(p_j)\to0$.  Put $d_j=\dist(p_j,\partial\Omega)$ and
choose
\[
 0<r_j<\min\{c_0\delta_j,d_j/2\},
\]
where $c_0$ is the constant in Lemma~\ref{lem:radial-test}, decreased if
necessary so that
$c_0<1/2$.  Then
\[
 \overline{B(p_j,r_j)}
 \subset\Omega\cap\mathbb R^2_{\mathrm{reg}}.
\]

Let $U,V$ be given by Lemma~\ref{lem:Jacobian-realization} and define
\[
 F^{(j)}(y)=U\left(\frac{y-p_j}{r_j}\right),
 \qquad
 G^{(j)}(y)=V\left(\frac{y-p_j}{r_j}\right).
\]
Let $D_j$ be the open Weyl chamber containing $B(p_j,r_j)$ and put
\[
 u^{(j)}(y)=
 \begin{cases}
   \rho_k(y)^{-1}F^{(j)}(y),&y\in D_j,\\
   0,&y\notin D_j,
 \end{cases}
\]
and, similarly, set
\[
 v^{(j)}(y)=
 \begin{cases}
   \rho_k(y)^{-1}G^{(j)}(y),&y\in D_j,\\
   0,&y\notin D_j.
 \end{cases}
\]
The supports stay away from the chamber walls, so these are smooth
compactly supported functions.  Lemma~\ref{lem:global-chamber-gauge} gives
\[
 \mathcal J_k(u^{(j)},v^{(j)})\,d\mu_k
 =\bigl(F^{(j)}_{y_1}G^{(j)}_{y_2}
       -F^{(j)}_{y_2}G^{(j)}_{y_1}\bigr)\,dy
 =r_j^{-2}q\left(\frac{y-p_j}{r_j}\right)dy.
\]
Thus Lemma~\ref{lem:radial-test} gives
\begin{align*}
 &\int_\Omega N_k(p_j,y)
 \mathcal J_k(u^{(j)},v^{(j)})(y)\,d\mu_k(y)\\
 &\qquad=\int_\Omega N_k(p_j,y)
 \bigl(F^{(j)}_{y_1}(y)G^{(j)}_{y_2}(y)
       -F^{(j)}_{y_2}(y)G^{(j)}_{y_1}(y)\bigr)\,dy
 \geq\frac{c}{w_k(p_j)}.
\end{align*}

It remains to check the Dunkl energies.  On $B(p_j,r_j)$,
\[
 |\langle\alpha,y\rangle|
 =|\alpha|\dist(y,H_\alpha)
 \geq |\alpha|(\delta_j-r_j)
 \geq \tfrac12|\alpha|\delta_j.
\]
The energy identity in Lemma~\ref{lem:global-chamber-gauge}, two-dimensional
scaling, and $r_j\leq c_0\delta_j$ therefore give
\begin{align*}
\|\nabla_k u^{(j)}\|_{L^2(d\mu_k)}^2
 &\leq \|\nabla U\|_{L^2(dx)}^2
    +C_k\delta_j^{-2}\|F^{(j)}\|_{L^2(dx)}^2\\
 &=\|\nabla U\|_{L^2(dx)}^2
    +C_k\frac{r_j^2}{\delta_j^2}\|U\|_{L^2(dx)}^2
 \leq C_{k,U}.
\end{align*}
The same estimate holds for $v^{(j)}$.  Also,
\[
 \|u^{(j)}\|_{L^2(d\mu_k)}=\|F^{(j)}\|_{L^2(dx)}
 =r_j\|U\|_{L^2(dx)},
\]
and the same identity holds for $v^{(j)}$.  Thus the full global
zero-extension Sobolev norms used below are uniformly bounded as well.
The potentials are continuous near $p_j$ by the local logarithmic kernel
estimate, so their essential $L^\infty(\Omega)$ norms are at least their
values at $p_j$.  This proves the proposition.
\end{proof}

\subsection{The intrinsic same-domain criterion}

For $u\in C_c^\infty(\Omega\setminus\mathcal H)$, let $\widetilde u$ be
its zero extension to $\mathbb R^2$.  Define
\[
 \|u\|_{W^{1,2}_{k,0}(\Omega)}
 =\left(
    \|\widetilde u\|_{L^2(d\mu_k)}^2
    +\sum_{i=1}^2
      \|T_{e_i}\widetilde u\|_{L^2(d\mu_k)}^2
  \right)^{1/2},
\]
and let $W^{1,2}_{k,0}(\Omega)$ be the completion of
$C_c^\infty(\Omega\setminus\mathcal H)$ in this norm.  The use of the
global zero extension is part of the definition, since the Dunkl
operators contain reflection terms.

\begin{theorem}
\label{thm:intrinsic-same-domain}
Let $\Omega\subset\mathbb R^2$ be a bounded domain.  For
$u,v\in C_c^\infty(\Omega\setminus\mathcal H)$, set
\[
 \mathcal W_{\Omega,k}(u,v)(x)
 =\int_\Omega N_k(x,y)
       \mathcal J_k(\widetilde u,\widetilde v)(y)\,d\mu_k(y).
\]
The map $\mathcal W_{\Omega,k}$ admits a bounded bilinear extension
\[
 \mathcal W_{\Omega,k}:
 W^{1,2}_{k,0}(\Omega)\times W^{1,2}_{k,0}(\Omega)
 \longrightarrow L^\infty(\Omega)
\]
if and only if
\[
 \dist(\overline\Omega,\mathcal H)>0.
\]
In that case,
\[
 \|\mathcal W_{\Omega,k}(u,v)\|_{L^\infty(\Omega)}
 \leq C_{\Omega,k}
 \|\nabla_k\widetilde u\|_{L^2(d\mu_k)}
 \|\nabla_k\widetilde v\|_{L^2(d\mu_k)}.
\]
\end{theorem}

\begin{proof}
Assume first that
$\dist(\overline\Omega,\mathcal H)>0$.  Since $\Omega$ is
connected, it lies in one open Weyl chamber $D$.  For test functions $u,v$,
put $F=\rho_k\widetilde u|_D$ and $G=\rho_k\widetilde v|_D$.  The global
chamber gauge lemma shows
that
\[
 \mathcal J_k(\widetilde u,\widetilde v)
\]
is supported in $\Omega$.  Hence the integral defining
$\mathcal W_{\Omega,k}$ is the global intrinsic Newton potential.  Apply
Proposition~\ref{prop:intrinsic-Newton-Wente} with
$K=\overline\Omega$.  This gives the stated bound on the test class.  For
approximating sequences, the bilinear identity
\begin{align*}
 &\mathcal W_{\Omega,k}(u^{(n)},v^{(n)})
 -\mathcal W_{\Omega,k}(u^{(m)},v^{(m)})\\
 &\quad=\Phi_k(\widetilde{u^{(n)}}-\widetilde{u^{(m)}},
                \widetilde{v^{(n)}})
  +\Phi_k(\widetilde{u^{(m)}},
          \widetilde{v^{(n)}}-\widetilde{v^{(m)}})
\end{align*}
together with Proposition~\ref{prop:intrinsic-Newton-Wente} shows that the
potentials form a Cauchy sequence in $L^\infty(\Omega)$.  Indeed, Cauchy
sequences in $W^{1,2}_{k,0}(\Omega)$ have uniformly bounded Dunkl-gradient
norms.  The limit is therefore independent of the approximations and gives
the asserted bilinear extension.

If $\dist(\overline\Omega,\mathcal H)=0$, the functions in
Proposition~\ref{prop:intrinsic-Wente-obstruction} have uniformly bounded
$W^{1,2}_{k,0}(\Omega)$ norms but their potentials have unbounded
$L^\infty(\Omega)$ norms.  Therefore no bounded bilinear extension can
exist.
\end{proof}

\subsection{The Euclidean same-domain corollary}

\begin{corollary}
\label{cor:Euclidean-same-domain}
Let $\Omega\subset\mathbb R^2$ be a bounded domain.  For
$F,G\in C_c^\infty(\Omega\setminus\mathcal H)$, set
\[
 W_\Omega(F,G)(x)
 =\int_\Omega N_k(x,y)
 \bigl(F_{y_1}(y)G_{y_2}(y)-F_{y_2}(y)G_{y_1}(y)\bigr)\,dy.
\]
By a bounded bilinear extension of $W_\Omega$ we mean a bounded bilinear map
\[
 W_\Omega:W_0^{1,2}(\Omega)\times W_0^{1,2}(\Omega)
 \longrightarrow L^\infty(\Omega)
\]
that agrees with $W_\Omega$ on
$C_c^\infty(\Omega\setminus\mathcal H)\times
C_c^\infty(\Omega\setminus\mathcal H)$.  Such an extension exists if and
only if
\[
 \dist(\overline\Omega,\mathcal H)>0.
\]
In that case,
\[
 \|W_\Omega(F,G)\|_{L^\infty(\Omega)}
 \leq C_{\Omega,k}
 \|\nabla F\|_{L^2(\Omega)}
 \|\nabla G\|_{L^2(\Omega)}
\]
for all $F,G\in W_0^{1,2}(\Omega)$.
\end{corollary}

\begin{proof}
If the distance is positive, the connected set $\Omega$ lies in one open
Weyl chamber.  Apply
Proposition~\ref{prop:Euclidean-source-Newton} with $U=\Omega$ and
$K=\overline\Omega$.

Conversely, use the functions $F^{(j)},G^{(j)}$ constructed in the proof of
Proposition~\ref{prop:intrinsic-Wente-obstruction}.  Their Euclidean
gradient norms are independent of $j$, while
\[
 \int_\Omega N_k(p_j,y)
 \bigl(F^{(j)}_{y_1}(y)G^{(j)}_{y_2}(y)
       -F^{(j)}_{y_2}(y)G^{(j)}_{y_1}(y)\bigr)\,dy
 \geq\frac{c}{w_k(p_j)}\longrightarrow\infty.
\]
Thus no such Euclidean-source bound is possible when the closure of
$\Omega$ meets the reflection arrangement.
\end{proof}

We finally present the endgame in the

\begin{proof} [Proof of Theorem~\ref{thm:intro-wente}]   This is a consequence of Theorem~\ref{thm:intrinsic-same-domain} and Corollary~\ref{cor:Euclidean-same-domain}.  
\end{proof}

\medskip

\section{A Brezis--Merle-type estimate}\label{sec:localized-newton}

\subsection{Localized Newton potentials in BMO}
Let $\Omega\Subset\mathbb{R}^{2}_{\mathrm{reg}}$ be a bounded open set. For
$f\in L^{1}(\Omega,d\mu_k)$, define
\[
N_\Omega f(x)
=\int_{\Omega}N_k(x,y)f(y)\,d\mu_k(y),
\qquad x\in\mathbb{R}^{2}.
\]
The regular-pole estimates imply that, for every compact $E\subset\mathbb{R}^{2}$,
\[
\sup_{y\in\overline\Omega}\int_E N_k(x,y)\,d\mu_k(x)<\infty.
\]
Fubini's theorem therefore shows that $N_\Omega f$ is defined almost
everywhere and belongs to $L^1_{\mathrm{loc}}(\mathbb{R}^{2},d\mu_k)$.

\begin{proposition}
	\label{prop5.1}
	There is $C_{\Omega,k}>0$ such that
	\[
	\|N_\Omega f\|_{\BMO(\mu_k)}
	\leq C_{\Omega,k}\|f\|_{L^1(\Omega,d\mu_k)}
	\]
	for every $f\in L^1(\Omega,d\mu_k)$.
\end{proposition}

\begin{proof}
	Let $B\subset\mathbb{R}^{2}$ be a Euclidean ball and put
	\[
	m_B(y)=\frac{1}{\mu_k(B)}\int_B N_k(x,y)\,d\mu_k(x).
	\]
	Set $c_B=\int_{\Omega}m_B(y)f(y)\,d\mu_k(y)$. Fubini's theorem and
	\eqref{eq:uniform-regular-poles-bmo} give
	\begin{align*}
		&\frac{1}{\mu_k(B)}\int_B
		|N_\Omega f(x)-c_B|\,d\mu_k(x)\\
		&\quad\leq \int_{\Omega}|f(y)|
		\left(\frac{1}{\mu_k(B)}
		\int_B|N_k(x,y)-m_B(y)|\,d\mu_k(x)\right)d\mu_k(y)\\
		&\quad\leq C_{\Omega,k}\|f\|_{L^1(\Omega,d\mu_k)}.
	\end{align*}
	Using
	\[
	\frac{1}{\mu_k(B)}\int_B|F-F_B|\,d\mu_k
	\leq 2\inf_{c\in\mathbb{R}}
	\frac{1}{\mu_k(B)}\int_B|F-c|\,d\mu_k
	\]
	and taking the supremum over $B$ proves the result.
\end{proof}

 \subsection{A Brezis--Merle estimate}
We next prove a Newton-kernel analogue of the estimate of Brezis--Merle~\cite{BrezisMerle}. Chanillo--Li~\cite[pp.~428--429]{ChanilloLi1992}
derive the corresponding elliptic estimate from Green-kernel BMO and the
John--Nirenberg inequality. In the present setting the uniform pointwise
logarithmic bound gives a shorter direct proof. The classical pointwise comparison
is provided by Kenig--Ni~\cite[Appendix, Theorem~A.4]{KenigNi}.

\begin{lemma}\label{lem:uniform-log-bound}
Let $K\Subset\mathbb{R}^{2}_{\mathrm{reg}}$. Then there are $C_K,r_K>0$ such that
\begin{equation}\label{eq:uniform-log-bound}
  0\leq N_k(x,y)
  \leq C_K\left(1+\log^+\frac{r_K}{|x-y|}\right),
  \qquad x,y\in K.
\end{equation}
\end{lemma}

\begin{proof}
Near the diagonal this is the upper bound in
Corollary~\ref{cor:regular-pole-log-bounds}, uniformly on compact subsets
of $\mathbb{R}^{2}_{\mathrm{reg}}$.  At a reflected orbit point distinct
from the pole it follows from
Lemma~\ref{lem:reflected-orbit-boundedness}.  Because
$K\Subset\mathbb{R}^{2}_{\mathrm{reg}}$, the orbit separation, the local
weight bounds, and the constants in
Corollary~\ref{cor:regular-pole-log-bounds} and
Lemma~\ref{lem:reflected-orbit-boundedness} can be chosen uniformly for
poles $y\in K$.
More explicitly, consider in $K\times K$ the diagonal and the finitely many
reflected graphs
\[
 \{(\sigma y,y):y\in K,\ \sigma y\in K\},
 \qquad \sigma\in G\setminus\{e\}.
\]
The diagonal has a uniform neighborhood controlled by
Corollary~\ref{cor:regular-pole-log-bounds}, and the nonidentity graphs have
uniform neighborhoods controlled by
Lemma~\ref{lem:reflected-orbit-boundedness}.  On the compact remainder the
orbit distance has a positive lower bound, so integrating the basic Gaussian
estimate gives a uniform bound.  A finite cover of $K\times K$ proves
\eqref{eq:uniform-log-bound}.
\end{proof}

\begin{theorem}\label{thm:brezis-merle-type}
Let $\Omega\Subset\mathbb{R}^{2}_{\mathrm{reg}}$. There are
$c_{\Omega,k},C_{\Omega,k}>0$ such that, for every nonzero
$f\in L^1(\Omega,d\mu_k)$,
\begin{equation}\label{eq:brezis-merle-type}
  \int_{\Omega}
  \exp\left(c_{\Omega,k}
  \frac{|N_\Omega f(x)|}
       {\|f\|_{L^1(\Omega,d\mu_k)}}\right)d\mu_k(x)
  \leq C_{\Omega,k}.
\end{equation}
\end{theorem}

\begin{proof}
Put $M=\|f\|_{L^1(\Omega,d\mu_k)}$ and
\[
  d\nu(y)=M^{-1}|f(y)|\,d\mu_k(y).
\]
Thus $\nu$ is a probability measure. Since $N_k\geq0$, Jensen's inequality gives
\[
  \exp\left(c\frac{|N_\Omega f(x)|}{M}\right)
  \leq\int_{\Omega}\exp\bigl(cN_k(x,y)\bigr)\,d\nu(y).
\]
Apply Lemma~\ref{lem:uniform-log-bound} with $K=\overline\Omega$ and choose
$c>0$ so that $cC_K<2$. Since $w_k$ is bounded above on $K$,
\eqref{eq:uniform-log-bound} gives
\[
  \sup_{y\in\overline\Omega}
  \int_{\Omega}e^{cN_k(x,y)}\,d\mu_k(x)
  \leq C\sup_{y\in\overline\Omega}
  \int_{\Omega}\bigl(1+|x-y|^{-cC_K}\bigr)\,dx
  \leq C_{\Omega,k}.
\]
Integrating the Jensen bound in $x$ and applying Fubini's theorem proves
\eqref{eq:brezis-merle-type}, with $c_{\Omega,k}=c$.
\end{proof}

\begin{corollary}\label{coro5.4}
Let $\Omega\Subset\mathbb{R}^{2}_{\mathrm{reg}}$ and
$f\in L^1(\Omega,d\mu_k)$. Then, for every $\beta>0$,
\[
  \exp\bigl(\beta|N_\Omega f|\bigr)
  \in L^1(\Omega,d\mu_k).
\]
\end{corollary}

\begin{proof}
Choose $g\in L^{\infty}(\Omega,d\mu_k)$ so that
\[
  \|f-g\|_{L^1(\Omega,d\mu_k)}<\frac{c_{\Omega,k}}{\beta}.
\]
Equivalently, this choice satisfies the strict condition
\[
  \beta\|f-g\|_{L^1(\Omega,d\mu_k)}<c_{\Omega,k}.
\]
The logarithmic majorant is integrable uniformly in its pole, and hence
\[
  \|N_\Omega g\|_{L^{\infty}(\Omega)}
  \leq C_{\Omega,k}\|g\|_{L^{\infty}(\Omega)}.
\]
If $f\neq g$, Theorem~\ref{thm:brezis-merle-type}, applied to $f-g$, shows that
$\exp\bigl(\beta|N_\Omega(f-g)|\bigr)$ is integrable; if $f=g$, this
is immediate. Since
\[
  e^{\beta|N_\Omega f|}
  \leq e^{\beta\|N_\Omega g\|_{\infty}}
        e^{\beta|N_\Omega(f-g)|},
\]
the conclusion follows.
\end{proof}

Finally, it is now easy to complete the

\begin{proof} [Proof of Theorem~\ref{thm:intro-brezis-merle}] From  Proposition~\ref{prop5.1}, Theorem~\ref{thm:brezis-merle-type}   and Corollary~\ref{coro5.4},
Theorem~\ref{thm:intro-brezis-merle} readily follows.

\end{proof}

\medskip

\noindent
{\bf Acknowledgements}:   Lixin Yan was supported  by National Key R$\&$D Program of China 2022YFA1005700  and 
by  NNSF of China (No. 12571111).

\end{document}